\documentclass[11pt,a4paper]{article}

\usepackage{authblk}

\makeatletter
\newcommand{\subjclass}[2][2010]{%
  \let\@oldtitle\@title%
  \gdef\@title{\@oldtitle\footnotetext{#1 \emph{Mathematics subject classification.} #2}}%
}
\newcommand{\keywords}[1]{%
  \let\@@oldtitle\@title%
  \gdef\@title{\@@oldtitle\footnotetext{\emph{Key words and phrases.} #1.}}%
}
\makeatother

\AtEndDocument{\bigskip{\footnotesize%
  \textit{E-mail address}, Xiaoqin Guo: \texttt{guoxq@ucmail.uc.edu} \par
  \addvspace{\medskipamount}
  \textit{E-mail address}, Hung Vinh Tran: \texttt{hung@math.wisc.edu}
}}

\usepackage{etoolbox}
\usepackage{comment,amsmath,amssymb,amsthm,bbm,mathtools,mathrsfs}
\usepackage{tikz,pgfplots,color,setspace,lmodern}
\usepackage[]{stix}
\usetikzlibrary{shapes, backgrounds,patterns,decorations.pathreplacing}
\usepackage{multicol,float}
\usepackage{makeidx}
\usepackage{hyperref, enumerate}
\pgfplotsset{compat=newest}
\makeindex

\def\var{{\rm Var}}

\def\supp{{\rm Supp}}

\def\XXint#1#2#3{{\setbox0=\hbox{$#1{#2#3}{\int}$ }
\vcenter{\hbox{$#2#3$ }}\kern-.6\wd0}}

\newcommand\error{\varepsilon}

\newcommand{\dd}{\mathrm{d}}
\newcommand{\by}[2]{\stackrel{#1}{#2}}
\newcommand{\B}{\mathbb{B}}

\newcommand{\R} {\mathbb{R}}
\newcommand{\Q} {\mathbb{Q}}
\newcommand{\Z} {\mathbb{Z}}
\newcommand{\nn}{\nonumber}
\newcommand{\N} {\mathbb{N}}
\newcommand{\dist} {\textnormal{dist}}

\newcommand{\tx}{\mathscr{H}}

\newcommand{\evp}[1]{\bar\omega^{#1}}
\newcommand{\vd}[1]{\partial'_{#1}}  
\newcommand{\rn}[1]{\mathrm{#1}}  
\newcommand{\tfa}{\zeta} 
\newcommand{\tfar}{\bar\tfa}
\newcommand{\tfb}{\xi}
\newcommand{\tfbr}{\bar\tfb}
\newcommand{\tfe}{\eta}
\newcommand{\tfer}{\bar\eta}

\DeclarePairedDelimiter{\abs}{\lvert}{\rvert}

\DeclarePairedDelimiter{\norm}{\lVert}{\rVert}
\providecommand{\Abs}[1]{\Bigr\lvert#1\Bigl\rvert}

\providecommand{\mb}[1]{\mathbb{#1}}

\providecommand{\ms}[1]{\mathscr{#1}}
\providecommand{\mf}[1]{\mathfrak{#1}}

\newcommand{\tr}{{\rm tr}}
\newcommand{\diag}{{\rm diag}}

\newtheorem{thmx}{Theorem}
\newtheorem{propx}[thmx]{Proposition}

\newtheorem{theorem}{Theorem}
\newtheorem{lemma}[theorem]{Lemma}
\newtheorem{corollary}[theorem]{Corollary}

\theoremstyle{definition}  
\newtheorem{definition}[theorem]{Definition}
\newtheorem{remark}[theorem]{Remark}

\begin{document}

\title{Stochastic integrability of heat-kernel bounds for random walks in a balanced random environment}

\author[1]{Xiaoqin Guo \thanks{The work of XG is supported by Simons Foundation through Collaboration Grant for Mathematicians \#852943.}}
\author[2]{Hung V. Tran  \thanks{HT is supported in part by NSF CAREER grant DMS-1843320 and a Vilas Faculty Early-Career Investigator Award.}}

\affil[1]
{
Department of Mathematical Sciences, 
University of Cincinnati , 2815 Commons Way, Cincinnati, OH 45221, USA}
\affil[2]
{
Department of Mathematics, 
University of Wisconsin Madison, 480 Lincoln  Drive, Madison, WI 53706, USA}



\maketitle

\begin{abstract}
We consider random walks in a balanced i.i.d.  random environment in $\mathbb Z^d$ for $d\geq 2$ and the corresponding discrete non-divergence form difference operators. 
We first obtain an exponential integrability of the heat kernel bounds.
We then prove the optimal diffusive decay of the semigroup generated by the heat kernel for $d \geq 3$.
As a consequence, we deduce a functional central limit theorem for the environment viewed from the particle.
\end{abstract}



\section{Introduction}\label{sec:intro}
In this article we consider random walks in a balanced i.i.d.   random environment in $\Z^d$ for $d\geq 2$. 
\subsection{Settings}
Let $\mb S_{d\times d}$ be the set of $d\times d$ positive-definite diagonal matrices. A map 
\[
\omega:\Z^d\to\mb S_{d\times d}
\] is called an {\it environment}. Denote the set of all environments by $\Omega$ and let $\mb P$ be a probability measure on $\Omega$ so that 
\[
\left\{\omega(x)=\mathrm{diag}[\omega_1(x),\ldots, \omega_d(x)], x\in\Z^d\right\}
\] are i.i.d.\,under $\mb P$. Expectation with respect to $\mb P$ is denoted by $\mb E$ or $E_{\mb P}$. 

Let $\{e_1,\ldots,e_d\}$ be the  canonical basis for $\R^d$.  
For a given function $u:\Z^d\to\R$ and $\omega\in\Omega$, define the non-divergence form difference operator
\begin{equation}\label{eq:def-nabla}
\tr(\omega(x)\nabla^2 u)=\sum_{i=1}^d\omega_i(x)[u(x+e_i)+u(x-e_i)-2u(x)],
\end{equation}
where $\nabla^2=\diag[\nabla^2_1,\ldots, \nabla_d^2]$, and 
$\nabla_i^2 u(x)=u(x+e_i)+u(x-e_i)-2u(x)$. 

For $r>0$, $y\in\R^d$ we let
\[
\B_r(y)=\left\{x\in\R^d: |x-y|<r\right\}, \quad
B_r(y)=\B_r(y)\cap\Z^d
\]
denote the continuous and discrete balls with center $y$ and radius $r$, respectively.
When $y=0$, we also write
$\B_r=\B_r(0)$ and $B_r=B_r(0)$. 
 For any $B\subset\Z^d$, its {\it discrete boundary} is the set
\[
\partial B:=\left\{z\in\Z^d\setminus B: \dist(z,x)=1 \text{ for some }x\in B\right\}.
\]
Let $\bar B=B\cup\partial B$. Note that with abuse of notation, whenever confusion does not occur,  we also use $\partial A$ and $\bar A$ to denote the usual continuous boundary and closure of $A\subset\R^d$, respectively.

 For $x\in\Z^d$, a {\it spatial shift} $\theta_x:\Omega\to\Omega$ is defined by 
 \[
 (\theta_x\omega)(\cdot)=\omega(x+\cdot).
 \]
In a random environment $\omega\in\Omega$, we consider the discrete elliptic Dirichlet problem
\begin{equation}\label{eq:elliptic-dirich}
\left\{
\begin{array}{lr}
\tfrac 12\tr(\omega\nabla^2u(x))=\frac{1}{R^2}f\left(\tfrac{x}{R}\right)\psi(\theta_x\omega) & x\in B_R,\\[5 pt]
u(x)=g\left(\tfrac{x}{|x|}\right) & x\in \partial B_R,
\end{array}
\right.
\end{equation}
where $f\in\R^{\B_1}, g\in\R^{\partial\B_1}$ are functions with nice enough regularity properties and $\psi\in\R^\Omega$ is bounded and satisfies suitable measurability conditions. Stochastic homogenization studies  (for $\mb P$-almost all $\omega$)  the convergence of $u$ to the solution $\bar u$ of a deterministic {\it effective equation}
\begin{equation}\label{eq:effective-ellip}
\left\{
\begin{array}{lr}
\tfrac 12\tr(\bar a D^2\bar{u})
=f\bar\psi &\text{ in }\B_1,\\ 
\bar u=g &\text{ on }\partial \B_1,
\end{array}
\right.
\end{equation}
as $R\to\infty$. Here $D^2 \bar{u}$ denotes the Hessian matrix of $\bar{u}$ and $\bar a=\bar a(\mb P)\in\mb S_{d\times d}$ and $\bar\psi=\bar\psi(\mb P,\psi)\in\R$ are {\it deterministic} and do not depend on the realization of the random environment (see the statement of Proposition \ref{thm:quant-homo} for formulas for $\bar{a}$ and $\bar{\psi}$).

The difference equation \eqref{eq:elliptic-dirich} is used to describe random walks in a random environment (RWRE) in $\Z^d$. To be specific, 
we set
\begin{equation}\label{def:omegaxei}
\omega(x,x\pm e_i):=\frac{\omega_i(x)}{2\tr\omega(x)} \quad  \text{ for } i=1,\ldots d,
\end{equation}
 and $\omega(x,y)=0$ if $|x-y|\neq 1$. Namely, we normalize $\omega$ to get a transition probability.  We remark that the configuration of $\{\omega(x,y):x,y\in\Z^d\}$ is also called a {\it balanced environment} in the literature \cite{L-82,GZ-12,BD-14}.
\begin{definition}\label{def:rwre-discrete}
For a fixed $\omega\in\Omega$, the random walk  $(X_n)_{n\ge 0}$ in the environment $\omega$ is a Markov chain  in $\Z^d$ with transition probability $P_\omega^x$ specified by
\begin{equation}\label{eq:def-RW}
P_\omega^x\left(X_{n+1}=z|X_n=y\right)=\omega(y,z).
\end{equation}
 \end{definition} 
The expectation with respect to $P_\omega^x$ is written as $E_\omega^x$. When the starting point of the random walk is $0$, we sometimes omit the superscript and simply write $P_\omega^0, E_\omega^0$ as $P_\omega$ and $E_\omega$, respectively.  Notice that for random walks $(X_n)$ in an environment $\omega$, 
\begin{equation}\label{def:omegabar}
\bar\omega^i=\theta_{X_i}\omega\in\Omega, \quad i\ge 0,
\end{equation}
is also a Markov chain, called the {\it environment viewed from the particle} process. With abuse of notation, we enlarge our probability space so that $P_\omega$ still denotes the joint law of the random walks and $(\bar\omega^i)_{i\ge 0}$. 

We also consider the {\it continuous-time} RWRE $(Y_t)$ on $\Z^d$.
\begin{definition}\label{def:rwre-continuous}
Let $(Y_t)_{t\ge 0}$ be the Markov process on $\Z^d$ with generator \begin{equation}\label{eq:def-of-L}
L_\omega u(x)=\sum_y\omega(x,y)[u(y)-u(x)]=\frac{1}{2\tr\omega(x)}\tr(\omega(x)\nabla^2 u).
\end{equation}
\end{definition}
With abuse of notation, we also let $P_\omega^x$ denote the quenched law of $(Y_t)$. If there is no ambiguity from the context, we also write, for $x,y\in\Z^d$, $n\in\Z, t\in\R$, the transition kernels of the discrete and continuous time walks as
\[
p_n^\omega(x,y)=P_\omega^x(X_n=y), \quad\text{ and }\quad
p_t^\omega(x,y)=P_\omega^x(Y_t=y),
\]
respectively. 

Both the discrete- and continuous-time RWRE share the same trajectory, and their behavior are very much the same.  
The solutions to the Dirichlet problem can be characterized using the discrete-time RWRE, whereas for the transition kernels it is easier to manipulate the continuous-time case where the derivatives have less cumbersome notation compared to theirs discrete counterparts.  Hence we will use both $(X_n)$ and $(Y_t)$ in our paper for convenience.

\subsection{Main assumptions}
We assume the following points throughout the paper.
\begin{enumerate}
\item[(A1)] $\left\{\omega(x), x\in \Z^d \right\}$ are i.i.d.\,under the probability measure $\mb P$.
\item[(A2)] $\frac{\omega}{\tr\omega}\ge 2\kappa{\rm I}$ for $\mb P$-almost every $\omega$ and some constant 
$\kappa\in(0,\tfrac{1}{2d}]$.
\item[(A3)] $\psi$ is a measurable function of the environment with the property that $\{\psi(\theta_x\omega):x\in\Z^d\}$ are i.i.d. under $\mb P$.
\end{enumerate} 
In this paper, we use $c, C$ to denote positive constants which may change from line to line but that only depend on the dimension $d$ and the ellipticity constant $\kappa$ unless otherwise stated. We write $A
\lesssim B$ if $A\le CB$, and $A\asymp B$ if both $A\lesssim B$ and $A\gtrsim B$ hold.

\subsection{Earlier results in the literature}

The following quenched central limit theorem (QCLT) was proved by Lawler \cite{L-82}, which is a discrete version of Papanicolaou,  Varadhan \cite{PV-82}.
\begin{thmx}\label{thm:QCLT}
Assume {\rm(A2)} and that law $\mb P$ of the environment is ergodic under spatial shifts $\{\theta_x:x\in\Z^d\}$. Then 
\begin{enumerate}[(i)]
\item There exists a probability measure $\mb Q\approx\mb P$ such that $(\evp{i})_{i\ge 0}$ is an ergodic (with respect to time shifts) sequence under law $\mb Q\times P_\omega$.
\item For $\mb P$-almost every $\omega$, the rescaled path $X_{n^2t}/n$ converges weakly (under law $P_\omega$) to a Brownian motion with covariance matrix $\bar a=E_\Q[\omega/\tr\omega]>0$. 
\end{enumerate}
\end{thmx}

QCLT for the balanced RWRE in static environments under weaker ellipticity assumptions can be found at \cite{GZ-12, BD-14}. For dynamic balanced random environment, QCLT was established in  \cite{DGR-15} and finer results concerning the local limit theorem and heat kernel estimates was obtained at \cite{DG-19}.  We refer to Zeitouni’s lecture notes \cite{OZ-04} for a comprehensive account of results and challenges in RWRE.

Denote the Radon-Nikodym derivative of $\Q$ with respect to $\mb P$ as
\begin{equation}\label{eq:def-rho}
\rho(\omega)=\dd\mb Q/\dd\mb P.
\end{equation}
For any $x\in\Z^d$ and finite set $A\subset\Z^d$, we write 
\[
\rho_\omega(x):=\rho(\theta_x\omega)
\quad \text{ and }
\quad
\rho_\omega(A)=\sum_{x\in A}\rho_\omega(x).
\]
It is known that (see, e.g., \cite{DG-19}) for $\mb P$-almost all $\omega$,  the measure $\rho_\omega(\cdot)$ on $\Z^d$ is the unique (up to a multiplicative constant) invariant measure for the RWRE $(X_n)_{n\ge 0}$.  In this sense, $\Q$ is the {\it steady state} for the environmental process.  To investigate the long term behavior of the RWRE and the homogenization of the corresponding diffusion equations, 
it is crucial to characterize the invariant measure $\rho_\omega$. 

Now that $\Q$ is the limiting ergodic measure of the environment,  it is expected that, as $t\to\infty$, $\psi(\bar\omega_t)\to E_\Q[\psi]$ almost surely for $\psi\in L^1(\mb P)$.  
(Recall the process $\bar\omega_t$ of the environment as viewed from the particle  in \eqref{def:omegabar}.)
When the balanced environment satisfies a finite range of dependence and $\psi$ is an $L^\infty(\mb P)$ local function, it is shown in \cite[Theorem~1.2]{GPT-19} that, with overwhelming $\mb P\otimes P_\omega$-probability,  the average $t^{-1}\int_0^t \psi(\omega_s)\dd s$ converges to $E_\Q[\psi]$ at an algebraic speed  $t^{-\alpha}$.  
For time-reversible random walks in random environment,  Kipnis and Varadhan \cite{KV-86} proved that the process of the environment as viewed from the particle has diffusive behavior.  Further, in the reversible setting where the walk is generated by divergence form operators $\nabla\cdot a(\omega)\nabla$,  algebraic rates for the decay of $E_\omega[\psi(\bar\omega_t)]$ are obtained in \cite{JM-11, GNO-15, PM-15}.  

One of our goals in this paper is to prove a diffusive behavior for the environmental process,  i.e., a CLT for $\tfrac1{\sqrt t}\int_0^t (\psi(\omega_s)-E_\Q\psi)\dd s$ and to investigate the decay rate of $E_\omega[\psi(\bar\omega_t)]-E_\Q\psi$.

As an important feature of the non-divergence form model, $\rho_\omega$ does not have deterministic upper and (nonzero) lower bounds. Moreover, the heat kernel $p_t^\omega(\cdot,\cdot)$ is not expected to have deterministic Gaussian bounds.
For ergodic balanced random environments,  the following stochastic bounds for the invariant measure $\rho_\omega$ and the heat kernel were proved in \cite{L-82,FS84,DG-19}.

For $r\ge 0, t>0$, denote by
\begin{equation}\label{eq:def-mf-h}
\mf h(r,t)=\frac{r^2}{r\vee t}+r\log(\frac{r}{t}\vee 1), \quad
r\ge 0, t>0.
\end{equation} 
\begin{thmx}\label{thm:recall-rho-estimates}
Assume {\rm(A2)} and that the environment is ergodic under spatial shifts $\{\theta_x:x\in\Z^d\}$. There exists a constant $p=p(d,\kappa)>0$ such that
\begin{enumerate}[(i)]
\item\label{item:old-rho-integ}
\[
\mb E[\rho^{d/(d-1)}]<\infty, \qquad \mb E[\rho^{-p}]<\infty;
\]
\item\label{item:recall-vd} $\mb P$-almost surely, for any $r>0$,
\[
\rho_\omega(B_r)\ge C\rho_\omega(B_{2r});
\]
\item\label{item:recall-hke} $\mb P$-almost surely, for all $x\in\Z^d$, $t>0$,
\begin{equation}\label{eq:hke-random}
\frac{c\rho_\omega(0)}{\rho_\omega(B_{\sqrt t})}e^{-C|x|^2/t}
\le 
p_t^\omega(x,0)
\le \frac{C\rho_\omega(0)}{\rho_\omega(B_{\sqrt t})}e^{-c\mf h(|x|,t)}.
\end{equation}
Moreover, for $x\in\Z^d$, $t>0$, 
\begin{align*}
\norm{p_t^\omega(0,x)}_{L^{(d+1)/d}(\mb P)}&\le \frac{C}{(t+1)^{d/2}}e^{-c\mf h(|x|,t)},\\
\norm{p_t^\omega(0,x)}_{L^{-p}(\mb P)}&\ge \frac{c}{(t+1)^{d/2}}e^{-C|x|^2/t}.
\end{align*}	
\end{enumerate}
\end{thmx}
The positive moment bound in \eqref{item:old-rho-integ} was obtained by Lawler \cite{L-82}, and the negative moment bound in \eqref{item:old-rho-integ} is a special case of the bound proved by Deuschel and the first named author  \cite[Theorem~11]{DG-19} in the more general  time-dependent ergodic environment setting. The volume-doubling property \eqref{item:recall-vd} was proved by Fabes and Stroock \cite[Lemma~2.0]{FS84}. The heat kernel bounds together with its integrability in the form of \eqref{item:recall-hke} were obtained in \cite[Theorem~11]{DG-19} for in the more general dynamic ergodic setting.

Roughly speaking, the term $\frac{\rho_\omega(0)}{\rho_\omega(B_{\sqrt t})}$ is the long term ratio between the time the RWRE visits the origin and the time it spends in the ball $B_{\sqrt t}$.  For 
the special deterministic environment $a\equiv I$, i.e.,  when the RWRE is a simple random walk,  we have $\rho\equiv 1$ and this constant is $Ct^{-d/2}$. However, in a random environment, one should not expect $\frac{\rho_\omega(0)t^{d/2}}{\rho_\omega(B_{\sqrt t})}$ to have deterministic bounds.  Hence, to understand the bounds of the invariant measure and the heat kernel, it is crucial to obtain their stochastic integrability. 

Although for general ergodic environment, positive and negative moment bounds for the invariant measure and the heat kernel were already obtained in Theorem~\ref{thm:recall-rho-estimates} \eqref{item:old-rho-integ}\eqref{item:recall-hke}, it is a natural question whether better mixing properties of the environment would yield better moment bounds for these quantities.

Our another goal in this paper is to show that,  in our i.i.d. balanced environment,  both $\rho_\omega$ and the heat kernel have positive and negative exponential moment bounds.

In the course of our proof of better moment bounds for the invariant measure and the heat kernels, the following quantitative homogenization result for the non-divergence form operator $L_\omega$ will be employed. 

\begin{propx}
\label{thm:quant-homo}
Assume {\rm(A1), (A2), (A3)}. 
Recall the measure $\Q$ in Theorem~\ref{thm:QCLT}. 
Suppose $g\in C^\alpha(\partial\B_1)$, $f\in C^{\alpha}(\B_1)$ for some $\alpha\in(0,1]$, and $\psi$ is a measurable function of $\omega(0)$ with $\norm{\psi/\tr\omega}_\infty <\infty$.
For any $q\in (0,d)$, there exist a random variable $\ms X_q=\ms X_q(\omega,d,\kappa)$ with $\mb E[\exp(c\ms X_q^d)]<\infty$, and a constant $\beta=\beta(d,\kappa,q)\in(0,1)$ such that for any $y\in B_{3R}$, the solution $u$ of 
\[
\left\{
\begin{array}{lr}
\frac{1}{2}\tr(\omega\nabla^2 u(x))=\frac{1}{R^2}f(\frac{x-y}{R})\psi(\theta_{x-y}\omega) & x\in B_R(y),\\
u(x)=g(\frac{x-y}{|x-y|}) & x\in\partial B_R(y)
\end{array}
\right.
\]
satisfies, with $A_1=\norm{f}_{C^{0,\alpha}(\B_1)}\norm{\tfrac{\psi}{\tr(\omega)}}_\infty+[g]_{C^{0,\alpha}(\partial\B_1)}$,
\begin{equation}\label{eq:homo-err}
\max_{x\in B_R(y)}\Abs{u(x)-\bar u(\tfrac{x-y}{R})}
\lesssim 
A_1(1+\ms X_q R^{-q/d})R^{-\alpha\beta}, 
\end{equation}
where $\bar u$ solves \eqref{eq:effective-ellip} with $\bar a=E_{\Q}[\omega/\tr\omega]>0$ and $\bar\psi=E_{\Q}[\psi/\tr\omega]$.
\end{propx}

\begin{remark}
Our Proposition~\ref{thm:quant-homo} in the above is a version of \cite[Theorem~1.5]{GPT-19}, which can be considered as a discrete version of Armstrong, Smart \cite[Theorem~1.2]{AS-14}.
We remark that, compared to the aforementioned results in \cite{AS-14,GPT-19}, a difference in Proposition~\ref{thm:quant-homo} is that the ball is allowed to be centered at any point $y\in B_{3R}$ whereas the random variable $\ms X$ stays the same.  This subtle feature will allow us to define a ``homogenization radius"  which will be useful later in our proof of the bounds for the Green functions.

The proof of Proposition~\ref{thm:quant-homo}, which is a small modification of that of \cite[Theorem~1.5]{GPT-19}, can be found in the Appendix.

In terms of the quantitative homogenization of non-divergence form operators in the PDE setting, Yurinski derived a second moment estimate of the homogenization error in \cite{Yur-88} for linear elliptic case, and Caffarelli, Souganidis \cite{CaSu10} proved a logarithmic convergence rate for the nonlinear elliptic case. Afterwards, Armstrong, Smart \cite{AS-14}, and Lin, Smart \cite{LS-15} achieved an algebraic convergence rate for fully nonlinear elliptic equations, and fully nonlinear parabolic equations, respectively. Armstrong, Lin \cite{AL-17} obtained quantitative estimates for the approximate corrector problems.
\end{remark}

\subsection{Main results}
For RWRE in an i.i.d. balanced, uniformly elliptic environment, 
we will establish natural bounds for both the invariant measure and the heat kernel (Theorem~\ref{thm:hk-bounds}) which possess both positive and negative exponential moments, greatly improving the stochastic integrability in Theorem~\ref{thm:hk-bounds} for the ergodic setting.
We then prove the optimal diffusive decay of the semigroup generated by the heat kernel for $d\ge 3$ in Theorem~\ref{thm:var_decay}. As consequences,  we obtain a functional central limit theorem (CLT) for the environment viewed from the particle process, and deduce the existence of a stationary corrector in dimension $d\ge 5$.


\begin{theorem}\label{thm:hk-bounds}
Assume {\rm(A1), (A2)}, and $d\ge 2$.
Let $s=s(d,\kappa)=2+\tfrac{1}{2\kappa}-d\ge 2$.
For any $\error>0$, there exists a random variable $\tx(\omega)=\tx(\omega,d,\kappa,\error)>0$ with 
$\mb E[\exp(c\tx^{d-\error})]<\infty$ such that the following properties hold.
\begin{enumerate}[(a)]
\item\label{item:rho-bounds} For $\mb P$-almost all $\omega$,
\[
c\tx^{-s}\le \rho(\omega)\le C\tx^{d-1}.
\]
In particular,  for any 
$q\in(-\tfrac{d}{s},\tfrac{d}{d-1})$, we have
\[
\mb E[\exp\left(c\rho^{q}\right)]<\infty.
\]
\item\label{item:rho-hke-expmmt}
Recall the function $\mf h$ in \eqref{eq:def-mf-h}. 
For any $r\ge 1$ and $\mb P$-almost all $\omega$,
\[
c\tx^{-s}\le
\frac{r^d\rho_\omega(0)}{\rho_\omega(B_r)}
\le
C\tx^{d-1}.
\]
\item\label{item:hke-expmmt}
For any $x\in\Z^d, t>0$, and $\mb P$-almost all $\omega$,
\begin{align*}
p_t^\omega(x,0)
&\le  
C\tx^{d-1}(1+t)^{-d/2}e^{-c\mf h(|x|,t)},\\
p_t^\omega(x,0)&\ge
c\tx^{-s}(1+t)^{-d/2}e^{-C|x|^2/t}.
\end{align*}	
\end{enumerate}
\end{theorem}

Recall the continuous time RWRE $(Y_t)_{t\ge 0}$ in Definition~\ref{def:rwre-continuous}. With abuse of notation, we still denote the process of the environment viewed from the particle as
\[
\evp{t}:=\theta_{Y_t}\omega.
\]

Following Gloria, Neukamm, Otto \cite{GNO-15}, for any measurable function $\zeta:\Omega\to\R$,  we define its {\it stationary extension} $\bar\zeta:\Z^d\times\Omega\to\R$ as
\[
\bar\zeta(x)=\bar \zeta(x;\omega):=\zeta(\theta_x\omega).
\]
Define the semigroup $P_t$, $t\ge 0$, on $\R^\Omega$ by
\[
P_t \zeta(\omega)=E_\omega^0[\zeta(\evp{t})]=\sum_z p_t^\omega(0,z)\bar \zeta(z;\omega).
\]

 The following theorem estimates the speed of decorrelation of the environmental process $\evp{t}$ from the original environment. It gives a rate $t^{-d/4}$ of decay for the semigroup, which is optimal. 
 A function $\zeta:\Omega\to\R$ is said to be {\it local} if it depends only on the environment $\{\omega(x):x\in S\}$ in a finite set $S\subset\Z^d$.  Such a set $S$ is called the {\it support} of $\zeta$ and denoted by $\supp(\zeta)$.
 
\begin{theorem}
\label{thm:var_decay}
Assume {\rm(A1), (A2)}, and $d\ge 3$. For any local measurable function $\zeta:\Omega\to\R$ with $\norm{\zeta}_\infty\le 1$ and $t\ge 0$, we have, for $C=C(d,\kappa,\#\supp(\zeta))$,
\begin{align}
&\var_{\Q}(P_t \zeta)\le C(1+t)^{-d/2};\label{eq:var_decay_q}\\
&\norm{P_t\zeta-E_\Q\zeta}_{L^1(\mb P)}+\norm{P_t\zeta-\mb E[P_t\zeta]}_{L^p(\mb P)}\le C_p(1+t)^{-d/4}
\quad \text{ for all }p\in(0,2). 
\label{eq:var_decay_p}
 \end{align}	 
\end{theorem} 

For divergence form operators, optimal diffusive decay of the semigroup generated by the heat kernel was obtained by Gloria, Neukamm, Otto \cite{GNO-15}, de Buyer, Mourrat \cite{PM-15}.
Our proof of Theorem \ref{thm:var_decay} follows the approach of \cite{GNO-15}, which uses an Efron-Stein type inequality and the Duhamel representation formula for the vertical derivative. 
However, unlike the divergence form setting \cite{GNO-15}, there are no deterministic Gaussian bounds for the heat kernel, and the steady state $\Q$ of the environment process $(\evp t)_{t\ge 0}$ is not the same as the original measure $\mb P$. To overcome these difficulties, our heat kernel estimates and the (negative and positive) moment bounds of the Radon-Nikodym derivative $\tfrac{\dd \mb P}{\dd\Q}$ in Theorem~\ref{thm:hk-bounds} play crucial roles. (Another feature of our non-divergence setting that worth mentioning is that there is no Caccioppoli estimates for $p>2$. See Lemma~\ref{lem:caccip}.)

As a consequence of Theorem~\ref{thm:var_decay} and the CLT of \cite{DL-03}, we obtain a functional CLT for an additive functional of the environmental process. Recall that $(\evp t)_{t\ge 0}$ is an ergodic sequence under the law $\Q\times P_\omega$ and time shifts. The following CLT says that, when $d\ge 3$, the fluctuation around the ergodic mean is approximately Gaussian under the diffusive rescaling. 

\begin{corollary}
\label{cor:FCLT}
Assume {\rm(A1), (A2)}. Let $d\ge 3$. For $\mb P$-almost all $\omega$ and any bounded measurable local function $\zeta$ of the environment with $E_\Q \zeta=0$, the $P_\omega$ law of 
\[
\frac{1}{\sqrt t}\int_0^t \zeta(\evp s)\dd s
\]
converges weakly to a Brownian motion with a deterministic diffusivity constant.
\end{corollary}

Another consequence of Theorem~\ref{thm:var_decay} is the existence of a stationary corrector in $d\ge 5$ for the non-divergence form homogenization problem \eqref{eq:elliptic-dirich}.
 
\begin{corollary}\label{cor:ex_sta_cor}
 Assume {\rm(A1), (A2)}. When $d\ge 5$, for any bounded local measurable function $\zeta:\Omega\to\R$, there exists $\phi:\Omega\to\R$ 
such that $\phi\in L^p(\mb P)$ for all $p\in(0,2)$, and for $\mb P$-almost all $\omega$, its 
 stationary extension $\bar\phi(x)=\phi(\theta_x\omega)$ solves
 \begin{equation}\label{eq:eq_stat_corr}
L_\omega \bar\phi(x)=\bar\zeta(x)-E_\Q[\zeta], \quad \text{ for all }x\in\Z^d.
 \end{equation}
 \end{corollary} 
 
  \begin{remark}
 The corrector $\bar \phi(x)$ plays a crucial role in the quantification of the homogenization error for non-divergence form operators.  
 We remark that our Corollary~\ref{cor:ex_sta_cor} is a weaker version of \cite[Theorem~7.1]{AL-17} where not only the existence of the stationary corrector was proved but a stretched exponential tail was also obtained.  
 Specifically, \cite{AL-17} derived such a result using  a large scale $C^{1,1}$ estimate for the homogenization problem. Although our Corollary~\ref{cor:ex_sta_cor} is an immediate consequence of the optimal diffusive decay rate,  it has much weaker stochastic integrability.
 
 In the classical periodic environment setting,  it is well-known that the existence of a stationary corrector implies that the optimal homogenization error of problem \eqref{eq:elliptic-dirich} is generically of scale $R^{-1}$.  Readers may refer to the classical books \cite{BLP,JKO} for the derivation of the rate in the periodic setting, and \cite{GTY-19, GST-22} for discussions on the optimality of the rates.  
  \end{remark}
 
To obtain the above results, we need to study properties of the Green functions.
For $d\ge 2, R\ge 1$, 
denote  the exit time from $B_R$ of the RWRE by  
\begin{equation}\label{def:tau}
\tau=\tau_R=\inf\{n\ge 0:X_n\notin B_R\}.
\end{equation}

\begin{definition}
\label{def:green}
For $R\ge 1$, $\omega\in\Omega$, $x\in \Z^d$, $S\subset\Z^d$,  the {\it Green function} $G_R(\cdot,\cdot)$ in the ball $B_R$ for the balanced random walk is defined by
\[
G_R(x, S)=G_R^\omega(x,S):=E_\omega^x\left[\sum_{n=0}^{\tau_R-1}\mathbbm{1}_{X_n\in S}\right], \quad x\in\bar B_R.
\]
We also write $G_R(x,y):=G_R^\omega(x,\{y\})$ and $G_R(x):=G(x,0)$.
\end{definition}

Note that for $d\ge 3$, by \cite[Theorem~1]{GZ-12}, the RWRE is transient, and so the Green function in the whole space
\[
G^\omega(x):=\lim_{R\to\infty}G_R(x)<\infty
\]
is well-defined for all $x\in\Z^d$, $\mb P$-almost surely.
Whereas, when $d=2$, the RWRE is recurrent, and thus the Green function in the whole $\Z^2$ is infinity. In this case, the {\it potential kernel} 
\begin{equation}
\label{eq:def-potential}
A(x)=A^\omega(x)=\sum_{n=0}^\infty [p^\omega_n(0,0)-p^\omega_n(x,0)],
\quad x\in\Z^2,
\end{equation}
is well-defined.  Note that $G$ and $A$ are both non-negative functions, and, for $x\in \Z^d$,
\[
L_\omega G(x)=-\mathbbm{1}_{x=0}, \quad \text{ if } d\ge 3, 
\]
and
\[
L_\omega A(x)=\mathbbm{1}_{x=0},\quad \text{ if } d=2.\]

\begin{theorem}\label{thm:Green-bound}
Assume {\rm(A1), (A2)}. 
For $r> 0$, let
\begin{equation}\label{eq:def_ur}
U(r):=\left\{
\begin{array}{lr}
-\log r & \quad d=2,\\
r^{2-d} &\quad d\ge 3.
\end{array}
\right.
\end{equation}
For any $\error>0$, there exists a random variable $\tx=\tx(\omega,d,\kappa,\error)>0$ with 
$\mb E[\exp(c\tx^{d-\error})]<\infty$ such that, 
for $\mb P$-almost surely, for all $x\in B_R$,
\[
\tx^{-s}[U(|x|+1)-U(R+2)]
\lesssim
G_R^\omega(x)\lesssim 
\tx^{d-1}[U(|x|+1)-U(R+2)],
\]
where $s=s(d,\kappa)=2+\tfrac{1}{2\kappa}-d\ge 2$.
\end{theorem}

We remark that an upper bound for the Green function of the approximate corrector (which is defined in the whole $\R^d$) was proved by Armstrong, Lin \cite[Proposition~4.1]{AL-17}.

Our proof of the bounds of $G_R^\omega$ follows the idea of Armstrong, Lin \cite[Proposition~4.1]{AL-17}.  
In Theorem~\ref{thm:Green-bound}, we apply their idea to obtain both upper and lower bounds for the Green function $G_R$ in  a finite region.

\begin{corollary}\label{cor:a_and_g}
Assume {\rm(A1), (A2)}.  Let $s$ be as in Theorem~\ref{thm:Green-bound}.  For any $\error>0$, there exists a random variable $\tx=\tx(\omega,d,\kappa,\error)>0$ with 
$\mb E[\exp(c\tx^{d-\error})]<\infty$ such that, 
$\mb P$-almost surely,  for all $x\in\Z^d$, 
\begin{align*}
\tx^{-s}\log(|x|+1)\lesssim
A^\omega(x)
\lesssim  \tx\log(|x|+1), &\text{ when }d=2;\\
\tx^{-s}(1+|x|)^{2-d}\lesssim
G^\omega(x)
\lesssim
\tx^{d-1}(1+|x|)^{2-d}, &\text{ when }d\ge 3.
\end{align*}
\end{corollary}

%

\section{Bounds of the Green function in a ball}

\newcounter{stepctr}[theorem]
\newcommand{\steps}{\stepcounter{stepctr}{\bf Step \thestepctr.}}

By the Markov property, $G_R(x,S)$ satisfies $G_R=0$ on $\partial B_R$ and
\begin{equation}\label{eq:LG_R}
L_\omega G_R(x,S)=-\mathbbm{1}_{x\in S},\quad x\in B_R. 
\end{equation}

Our proof of the bounds of the Green function $G_R$ (Theorem~\ref{thm:Green-bound}) follows the idea of Armstrong-Lin \cite[Proposition~4.1]{AL-17}.  
The idea,  which we learn from \cite{AL-17},  is through the comparison of $G_R$ and test functions, and it is explained as follows.  

Let us call a function $u:\Z^d\to\R$ {\it $\omega$-harmonic} on $A\subset\Z^d$ if $L_\omega u(y)=0$ for $y\in A$.  Clearly, the Green function $G_R(x,0)$ is $\omega$-harmonic on $B_R\setminus\{0\}$.

To obtain the upper bound,  we construct a function $h$ which is almost $\omega$-harmonic away from the origin and with (almost) zero boundary values,  so that $G_R-h$ is sub-harmonic at places that are either close to the origin or the boundary of $B_R$.  
As a result,  if $(G_R-h)(x_0)=\max_{B_R}(G_R-h)$ were positive,  then the maximum principle forces the maximizer $x_0$ to be sufficiently far away from both the origin and the boundary.  This allows enough space for the homogenization to occur around $x_0$, i.e.,  $h$ is close to its continuous harmonic counterpart up to an algebraic error.  On the other hand,  
by the maximal principle, the $\omega$-harmonic counterpart of $G_R-h$ (which is an algebraic error away from $G_R-h$) cannot achieve its maximum over the ball $\bar B_{|x_0|/2}(x_0)$ in the center. This  would contradict the assumption that the maximizer is $x_0$,  
if we can exploit the fact that $h$ is strictly $\omega$-superharmonic to give it enough room to absorb the algebraic homogenization error.

The proof of the lower bound follows similar philosophy.

Note that \cite[Proposition~4.1]{AL-17} only deals with the Green function of the ``approximate corrector" which is defined on the whole $\R^d$,  while our Green function corresponds to the original non-divergence form operator  within a finite ball $B_R$.  Hence, in our case,  the challenge also lies in the construction of test functions so that they have the desired 
boundary values and concavity near the discrete boundary.

The following Lemmas contain properties of some deterministic functions that will be useful in our construction of the test functions in the next subsections.

\begin{lemma}\label{lem:tfab}
Let $\delta=\beta/2$, where $\beta=\beta(d,\kappa,q_\error)$ is as defined in Proposition~\ref{thm:quant-homo}. Define  $\tfar, \tfbr:(0,\infty)\to\R$ as
\[
\tfar(r)=
\left\{
\begin{array}{lr}
-(\log r)\exp(r^{-\delta}/\delta) & d=2\\
r^{2-d}\exp(-r^{-\delta}/\delta) & d\ge 3,
\end{array}
\right.
\]
\[
\tfbr(r)=\left\{
\begin{array}{lr}
-(\log r)\exp(-r^{-\delta}/\delta) & d=2\\
r^{2-d}\exp(r^{-\delta}/\delta) & d\ge 3.
\end{array}
\right.
\] 
Define two functions $\tfa, \tfb:\R^d\setminus\{0\}\to\R$ as 
\[
\tfa(y)=\tfar(|y|), 
\quad \text{ and }\quad
\tfb(y)=\tfbr(|y|),
\quad y\neq 0.
\]
Then, the following statements hold.
\begin{enumerate}[(i)]
\item\label{item:tfab_der} $\tfar, \tfbr$ are decreasing functions on $(C,\infty)$. Moreover, for $r\ge C$, 
\[
-\tfar'(r)\asymp 
r^{1-d},
\quad \text{ and }\quad
0.5r^{1-d}\le -\tfbr'(r)\le (d-0.5)r^{1-d}.
\]
\item\label{item:tfab_laplace} For $|y|\ge C$, we have 
\[
-\Delta\tfa(y)\ge |y|^{-(2+\delta)}|\zeta(y)|,
\quad \text{ and }\quad
-\Delta\tfb(y)\le -|y|^{-(2+\delta)}\abs{\tfb(y)}.
\]
\item\label{item:tfab_ck} For $|y|\ge 2$ and $k\in\N$, there exists $C=C(k,d)$ such that
\[
|D^k\tfa(y)|\le C|y|^{-k}|\tfa(y)|,
\quad \text{ and }\quad
|D^k\tfb(y)|\le 
C|y|^{-k}\abs{\tfb(y)}.
\]
\end{enumerate}
\end{lemma}

\begin{lemma}\label{lem:test-exponential}
There exist constants $\alpha_0\in(0,1)$ and $A_0\ge 1$ depending on $\kappa$ such that, for any $\alpha\in(0,\alpha_0)$, $A\ge A_0$, 
\begin{align}
&L_\omega(e^{-2\alpha|x|/R})\le 0 \quad \text{ in $B_R\setminus B_{R/2}$, when }R\ge A_0;\label{eq:20200203-1}\\
&L_\omega(e^{-A|x|^2})\ge -\mathbbm{1}_{x=0}, \quad x\in\Z^d;\label{eq:L-e^x2}\\
&L_\omega(e^{-A|x|^2/R^2})>0, \quad x\in B_R\setminus B_{R/2}, \text{ when }R\ge A^2.\label{eq:L-e^x2/R2}
\end{align}
\end{lemma}
The proof of Lemma~\ref{lem:test-exponential} is in Section~\ref{asection:test-fun} of the Appendix.

\subsection{Upper bounds of Green's functions}
Recall $\ms X_q$ in Proposition~\ref{thm:quant-homo}.
For any $\error\in(0,1)$, we write
\[R_0=R_0(\omega,d,\kappa,\error):=\ms X_{d-\error}^{d/(d-\error)}+K, \]
where $K$ is a sufficiently large constant depending on $(d,\kappa)$, and denote the exit time from $B_{R_0}$ as 
\begin{equation}\label{eq:def_r0s0}
s_0=\min\{n\ge 0:X_n\notin B_{R_0}\}.
\end{equation}
Note that $R_0$ plays the role of a "homogenization radius" in the sense that for all $R\ge  R_0$ and $y\in B_{3R}$,  the upper bound in \eqref{eq:homo-err} can be replaced by the algebraic term $CA_1 R^{-\alpha\beta}$.

Let $\alpha=\alpha(d,\kappa)>0$ be a constant to be determined in Lemma~\ref{lem:h_properties}, and set
\[
C_{\alpha,R}:=\frac{[\tfar(R/2)-\tfar(R)]R^{d-2}}{e^{-\alpha+2\alpha/R}-e^{-2\alpha}}
\asymp
\alpha^{-1}, \quad\text{when } R\ge R_0.
\]
\begin{definition}\label{def:h}
Let $\tfar, \tfa$ be as in Lemma~\ref{lem:tfab}.  For any fixed $R\ge 4R_0$, 
we define a function $h:\bar B_R\to[0,\infty)$ by
\[
h(x)=\left\{
\begin{array}{lr}
h_1(x), &x\in B_{R_0},\\
h_2(x), &x\in  B_{R/2}\setminus B_{R_0},\\
h_3(x), &x\in\bar B_{R}\setminus B_{R/2},
\end{array}
\right.
\]
where the functions $h_1, h_2, h_3$, are defined as below
\begin{align*}
h_2(y)&=R_0^{d-1}\left[
(\alpha^{-1}-1)(\tfar(R/2)-\tfar(R))+\tfa(y)-\tfar(R)
\right] ,\quad y\in\R^d\setminus\{0\},\\
h_1(x)&=E^x_\omega[h_2(X_{s_0})+|X_{s_0}|-|x|],\quad x\in\bar B_{R_0},\\
h_3(y)&=R_0^{d-1}\alpha^{-1}C_{\alpha,R}R^{2-d}
 [e^{-2\alpha(|y|-1)/R}-e^{-2\alpha}],\quad y\in\R^d.
\end{align*}	
\end{definition}

\begin{lemma}\label{lem:h_properties}
When $R\ge 4R_0$, there exists a constant $\alpha>0$ such that the functions $h_1,h_2,h_3, h$ given in Definition~\ref{def:h} have the following properties.
\begin{enumerate}[(a)]
\item\label{item:2d-h-1} $L_\omega h_1(x)=-L_\omega(|x|)\le -\mathbbm{1}_{x=0}$ for $x\in B_{R_0}$;
\item\label{item:2d-h-2} $h_1=h_2$ on $\partial B_{R_0}$, and $h_2=h_3$ on $\partial \B_{R/2}$;
\item\label{item:2d-h-3} $h_2\ge h_1$ in $B_{R_0}\setminus B_{R_0/2}$. 
\item\label{item:2d-h-4} $h_2\ge h_3$ in $\B_R\setminus\B_{R/2}$, and $h_2\le h_3$ in $\B_{R/2}\setminus \B_{R/2-1}$.
\item\label{item:2d-h-5} 
 $L_\omega h_3\le 0$ in $B_R\setminus B_{R/2}$.
\end{enumerate}
\end{lemma}
\begin{proof}
(a) and (b) are obvious. To see \eqref{item:2d-h-3} , 
note that $h_2(x)-h_1(x)=E_\omega^x[f(|x|)-f(|X_{s_0}|)]$, where $f(r)=r+R_0^{d-1}\tfar(r)$. Since $R_0\ge K$, by Lemma~\ref{lem:tfab}\eqref{item:tfab_der}, taking $K$  sufficiently large, $f(r)$ is a decreasing function for $r\in [R_0/2,R_0+1]$.  

Next, we will prove \eqref{item:2d-h-4} .
 Indeed, we can write, for $y\in\R^d\setminus\{0\}$,
 \[
h_2(y)-h_3(y)=
R_0^{d-1}a(|y|)+A(R_0,R),
\] 
where $A(R_0,R)$ is a constant, and $a(r)=\tfar(r)-\alpha^{-1}C_{\alpha,R}R^{2-d}e^{-2\alpha(r-1)/R}$. For $r\in[R/2-1, R]$, by Lemma~\ref{lem:tfab}\eqref{item:tfab_der}, taking $\alpha>0$ sufficiently small, 
we have 
\[
a'(r)
\ge -Cr^{1-d}+2C_{\alpha,R} e^{-2\alpha} R^{1-d}
\ge -Cr^{1-d}+c\alpha^{-1} R^{1-d}
\ge 0.
\]
Hence, $h_2-h_3$ is radially increasing in $\B_{R}\setminus \B_{R/2-1}$. Item \eqref{item:2d-h-4} then follows from the fact that $h_2-h_3=0$ on $\partial\B_{R/2}$.  

Item \eqref{item:2d-h-5} is a consequence of \eqref{eq:20200203-1} in Lemma~\ref{lem:test-exponential}.
\end{proof}

\begin{proof}
[Proof of the upper bound in  Theorem~\ref{thm:Green-bound}]
For $0< a<b$ and $d\ge 2$, we have an elementary inequality
\begin{equation}
\label{eq:20200607-3}
b-a\le b^{d-1}(U(a)-U(b)).
\end{equation}

When $R\in(1, 4R_0]$, note that $L_\omega[G_R(x)-(R+1-|x|)]\ge 0$, and $G_R=0\le R+1-|x|$ on $\partial B_R$. By the maximum principle, we have, for $x\in B_R$,
\begin{equation*}
G_R(x)\le R+1-|x|\by{\eqref{eq:20200607-3}}{\le} (R+2)^{d-1}[U(|x|+1)-U(R+2)].
\end{equation*}
Hence the upper bound in  Theorem~\ref{thm:Green-bound} holds when $R\in(1, 4R_0]$. 
It remains to consider the case 
$R>4R_0$. 

First, we will prove via contradiction that 
\begin{equation}
\label{eq:2d-green<h}
G_R\le h  \quad \text{ in }  \bar B_R.
\end{equation}
 Assume by contradiction that \eqref{eq:2d-green<h} fails, i.e., $\max_{\bar B_R}(G_R-h)>0$. 
By Lemma~\ref{lem:h_properties}\eqref{item:2d-h-1}, $L_\omega(G_R-h)\ge 0$ in $B_{R_0}$  and so $\max_{\bar B_R}(G_R-h)$ is achieved outside of $B_{R_0}$. 
Further, note that $(G_R-h)|_{\partial B_R}=(-h_3)|_{\partial B_R}\le 0$.
By Lemma~\ref{lem:h_properties}\eqref{item:2d-h-5}, $L_\omega(G_R-h_3)\ge 0$ in $B_R\setminus B_{R/2}$, and so, by the maximum principle and Lemma~\ref{lem:h_properties}\eqref{item:2d-h-4}, 
\[
\max_{\bar B_R\setminus B_{R/2}}(G_R-h)
\le \max_{\partial(B_R\setminus B_{R/2})}(G_R-h_3)
\le 0\vee \max_{B_{R/2}\setminus B_{R_0}}(G_R-h).
\]
Hence, if $\max_{\bar B_R}(G_R-h)>0$, then there exists $x_0\in  B_{R/2}\setminus B_{R_0}$ so that 
\[
(G_R-h)(x_0)=\max_{\bar B_R}(G_R-h)>0.
\]
Since $x_0\in  B_{R/2}\setminus B_{R_0}$, by Lemma~\ref{lem:h_properties}\eqref{item:2d-h-3}\eqref{item:2d-h-4}, 
\[
(G_R-h_2)(x_0)\ge 
\max_{\bar B_{|x_0|/2}(x_0)}(G_R-h_2),
\] 
which is equivalent to
\begin{equation}\label{eq:20200204-4}
(G_R-R_0^{d-1}\zeta)(x_0)\ge \max_{\bar B_{|x_0|/2}(x_0)}(G_R-R_0^{d-1}\zeta).
\end{equation}

Without loss of generality, assume that $\bar a=I$, and set
\[
\tilde \zeta(y):=\tfa(y)+c|x_0|^{-(2+\delta)}\abs{\tfa(x_0)}|y-x_0|^2, 
\quad
y\in \R^d\setminus\{0\},
\]
where $c>0$ is chosen so that (by Lemma~\ref{lem:tfab}\eqref{item:tfab_laplace}) $\Delta\tilde{\zeta}(y)\le 0$ for $y\in \B_{|x_0|/2}(x_0)$. Then (by Lemma~\ref{lem:tfab}\eqref{item:tfab_ck}) $|D\tilde{\zeta}|\le C|x_0|^{-1}|\zeta(x_0)|$ in $\B_{1+0.5|x_0|}(x_0)$,  and
\begin{equation}\label{eq:20200204-3}
(G_R-R_0^{d-1}\tilde\zeta)(x_0)
\stackrel{\eqref{eq:20200204-4}}{\ge}
 \max_{\partial B_{|x_0|/2}(x_0)}(G_R-R_0^{d-1}\tilde{\zeta})+CR_0^{d-1}|x_0|^{-\delta}\abs{\tfa(x_0)}.
\end{equation}
Let $\bar v:\bar\B_1\to\R$ and $v:\bar B_{|x_0|/2}(x_0)\to\R$ be the solutions of (Here $\bar a=I$.)
\[
\left\{
\begin{array}{lr}
\tr(\bar aD^2\bar v)=\Delta \bar v=0 & x\in \B_1\\
\bar v(x)=R_0^{d-1}\tilde\zeta(x_0+\tfrac{|x_0|}{2}x) & x\in\partial\B_1,
\end{array}
\right.
\]
and
\[
\left\{
\begin{array}{lr}
L_\omega v(x)=0 & x\in B_{|x_0|/2}(x_0)\\
v(x)=R_0^{d-1}\bar v(\tfrac{x-x_0}{|x-x_0|}) & x\in\partial B_{|x_0|/2}(x_0).
\end{array}
\right.
\]

We will show that $v$ can be controlled by $R_0^{d-1}\tilde{\zeta}$  both on the boundary and inside of $B_{|x_0|/2}(x_0)$. Indeed,  for $x\in\partial B_{|x_0|/2}(x_0)$, 
\begin{align}\label{eq:220925-1}
\abs{v(x)-R_0^{d-1}\tilde{\zeta}(x)}
&=
R_0^{d-1}\Abs{\tilde\zeta(x_0+\tfrac{|x_0|(x-x_0)}{2|x-x_0|}) -\tilde{\zeta}(x)}\nn\\
&\le R_0^{d-1}
\sup_{\bar\B_{1+|x_0|/2}(x_0)}
\abs{D\tilde{\zeta}}\nn\\
&\le 
C R_0^{d-1}|x_0|^{-1}\abs{\tfa(x_0)}.
\end{align}
For $x\in B_{|x_0|/2}(x_0)$,  applying Proposition~\ref{thm:quant-homo} to the case $\alpha=1$,  there exists $\beta=\beta(d,\kappa,\varepsilon)\in(0,1)$ such that
\begin{align}\label{eq:20200607-2}
v(x)
&\le 
\bar v(\tfrac{x-x_0}{|x_0|/2})+C|x_0|^{-\beta}R_0^{d-1}\sup_{y\in\partial\B_1}\Abs{D\tilde\zeta(x_0+\tfrac{|x_0|}{2}y)}\nn\\
&\le 
\bar v(\tfrac{x-x_0}{|x_0|/2})+CR_0^{d-1}|x_0|^{-\beta}\abs{\tfa(x_0)}.
\end{align}
Furthermore,  using the fact that $\Delta\tilde\zeta(x_0+\tfrac{|x_0|}{2}x)\le 0$ for $x\in\B_1$, we get $\bar v(x)\le R_0^{d-1}\tilde\zeta(x_0+\tfrac{|x_0|}{2}x)$ in $\B_1$.  This, together with \eqref{eq:20200607-2}, yields, for $x\in B_{|x_0|/2}(x_0)$,
\begin{equation}\label{eq:220925-2}
v(x)\le 
R_0^{d-1}\tilde\zeta(x)+CR_0^{d-1}|x_0|^{-\beta}\abs{\tfa(x_0)}.
\end{equation}
Notice that $(G_R-v)$ is an $\omega$-harmonic function on $B_{|x_0|/2}(x_0)$,  and so
\[
\max_{B_{|x_0|/2}(x_0)}(G_R-v)\le \max_{\partial B_{|x_0|/2}(x_0)}(G_R-v).
\]
Therefore, combining this inequality and \eqref{eq:220925-1}, \eqref{eq:220925-2}, we get
\[
\max_{B_{|x_0|/2}(x_0)}(G_R-R_0^{d-1}\tilde\zeta)
\le 
\max_{\partial B_{|x_0|/2}(x_0)}(G_R-R_0^{d-1}\tilde\zeta)+CR_0^{d-1}|x_0|^{-\beta}\abs{\tfa(x_0)}
\]
which contradicts \eqref{eq:20200204-3},  since $|x_0|\in(R_0,R)$,   $\delta=\beta/2$ by definition in Lemma~\ref{lem:tfab}, and $R_0\ge K$ is chosen to be sufficiently large. 
Inequality \eqref{eq:2d-green<h} is proved.

Finally, when $x\in  B_{R/2}\setminus B_{R_0}$,  we have $G_R(x)\by{\eqref{eq:2d-green<h}}{\le} h_2(x)$, and, by Lemma~\ref{lem:tfab}\eqref{item:tfab_der}, 
\begin{align}\label{eq:G-h-in-ring}
h_2(x)
&\le R_0^{d-1}\alpha^{-1}[\tfar(|x|)-\tfar(R)]\nn\\
&\lesssim R_0^{d-1}\int_{|x|}^{R}(-\tfar)'(r)\dd r
\nn\\&
\lesssim R_0^{d-1}\int_{|x|}^R r^{1-d}\dd r
\lesssim R_0^{d-1}[U(|x|)-U(R)].
\end{align}
When $x\in B_{R_0}\setminus\{0\}$, 
\begin{align*}
G_R(x)
&\stackrel{\eqref{eq:2d-green<h}}{\le} h_1(x)=
E_\omega^x [h_2(X_{s_0})+|X_{s_0}|-|x|]\\
&\by{\eqref{eq:G-h-in-ring}, \eqref{eq:20200607-3}}\lesssim 
R_0^{d-1}E_\omega^x[U(|X_{s_0}|)-U(R)+U(|x|)-U(|X_{s_0}|)]\\
&=R_0^{d-1}[U(|x|)-U(R)].
\end{align*}
Note that, for $|x|\ge 1$, $U(|x|)-U(R)\lesssim U(|x|+1)-U(R+2)$.

When $x\in B_R\setminus  B_{R/2}$, 
\begin{align*}
G_R&\le h_3\le CR_0^{d-1}R^{2-d}(e^{-2\alpha(|x|-1)/R}-e^{-2\alpha})\\
&\lesssim
R_0^{d-1}R^{2-d}(1-\tfrac{|x|-1}{R})=R_0^{d-1}R^{1-d}(R+1-|x|)\\
&\by{\eqref{eq:20200607-3}}\lesssim
R_0^{d-1}[U(|x|+1)-U(R+2)].
\end{align*}
The upper bound in Theorem~\ref{thm:Green-bound} is proved by putting $\tx=R_0$.
\end{proof}

\subsection{Lower bounds of Green's functions}

The proof of the lower bound of Theorem~\ref{thm:Green-bound}, which is similar to that of Theorem~\ref{thm:Green-bound}, is via comparing $G_R$ to appropriate test functions. However, unlike the Green function in the whole space, $G_R$ is defined only in bounded region, and so the test functions should be carefully designed to capture the behavior of $G_R$ near the boundary.


\begin{lemma}\label{lem:eta}
Define $\tfer:\R\to(0,\infty)$ as $\tfer(r)=(1+r^2)^{-\theta}$, where
\begin{equation}\label{put:theta}
\theta:=1/(4\kappa)\ge d/2.
\end{equation}
Define $\tfe:\R^d\to\R$ as
\[
\tfe(y)=\tfer(|y|).
\]
There exists a constant $C_0=C_0(d,\kappa)>0$ such that, for $x\in\Z^d$,
\[
L_\omega\eta(x)\ge -\mathbbm{1}_{x\in B_{C_0\theta^2}}.
\]
\end{lemma}
The proof of Lemma~\ref{lem:eta} is in Section~\ref{asection:test-fun} of the Appendix.\\

Let $\gamma=\gamma(\kappa)>0$ be a large constant to be determined, and set
\[
C_{\gamma,R}:=\frac{[\tfbr(R/2)-\tfbr(R)]R^{d-2}}{e^{-\gamma/4}-e^{-\gamma}}
\asymp 
e^{\gamma/4}
\quad
\text{ when }R\ge R_0.
\]
\begin{definition}\label{def:ell}
Recall $\tfb,\tfbr, \tfe,\tfer, \theta$ be in Lemma~\ref{lem:tfab} and Lemma~\ref{lem:eta}. 
 For any fixed $R\ge 4R_0$, we define three functions $\ell_i$, $i=1,2,3$, as
\begin{align*}
\ell_2(y)
&=R_0^{d-2-2\theta}\left[
(\gamma^{-2}-1)(\tfbr(R/2)-\tfbr(R))+\tfb(y)-\tfbr(R))
\right], \quad y\in\R^d\setminus\{0\};\\
\ell_1(x)
&=E_\omega^x[\ell_2(X_{s_0})+\eta(x)-\eta(X_{s_0})], \quad x\in \bar  B_{R_0};\\
\ell_3(y)&=R_0^{d-2-2\theta}\gamma^{-2} C_{\gamma,R} R^{2-d}(e^{-\gamma |y|^2/R^2}-e^{-\gamma}), \quad y\in\R^d,
\end{align*}
Also, for $R\ge 4R_0$, we define a function $\ell:\bar B_R\to\R$ by 
\[
\ell(x)
=\left\{
\begin{array}{lr}
\ell_1(x), & x\in B_{R_0},\\
\ell_2(x), &x\in B_{R/2}\setminus B_{R_0},\\
\ell_3(x), &x\in \bar B_R\setminus  B_{R/2}.
\end{array}
\right.
\]
\end{definition}

\begin{lemma}
\label{lem:ell_properties}
When $R\ge 4R_0$, there exists a constant $\gamma>0$ such that the functions $\ell_1,\ell_2,\ell_3,\ell$ given in Definition~\ref{def:ell} have the following properties.
\begin{enumerate}[(a)]
\item\label{item:ell-2d-1} $L_\omega \ell_1=L_\omega\eta\ge-\mathbbm{1}_{x\in B_{C_0\theta^2}}$ for $x\in B_{R_0}$;
\item\label{item:ell-2d-2} $\ell_1=\ell_2$ on $\partial B_{R_0}$, and $\ell_2=\ell_3$ on $\partial\B_{R/2}$;
\item\label{item:ell-2d-3} $\ell_2\le \ell_1$ in $B_{R_0}\setminus B_{R_0/2}$. 
\item\label{item:ell-2d-4}
$\ell_2\le \ell_3$ in $\B_R\setminus\B_{R/2}$, and $\ell_2\ge\ell_3$ in $\B_{R/2}\setminus\B_{0.5R-2}$. 
\item\label{item:ell-2d-5}
$L_\omega\ell_3\ge 0$ in $B_R\setminus B_{R/2}$. 
\end{enumerate}
\end{lemma}
\begin{proof}
\eqref{item:ell-2d-1} follows from Lemma~\ref{lem:eta}, and \eqref{item:ell-2d-2} follows from definition. To see \eqref{item:ell-2d-3}, note that $\ell_1-\ell_2=E_\omega^x[f(|X_{s_0}|)-f(|x|)]$, where
$f(r)=R_0^{d-2-2\theta}\tfbr(r)-\tfer(r)$. Since $R_0\ge K$, by taking $K$ sufficiently large and by Lemma~\ref{lem:tfab}\eqref{item:tfab_der}, we have
\begin{align*}
f'(r)
&\ge -(d-0.5) R_0^{d-2-2\theta}r^{1-d}+2\theta(1+\tfrac{1}{r^2})^{-\theta-1} r^{d-2-2\theta}\\
&\stackrel{\eqref{put:theta}}{\ge }
(d-0.5)(r^{d-2-2\theta}-R_0^{d-2-2\theta})\ge 0
\end{align*}
and so $f(r)$ is decreasing for $r\in[R_0/2, R_0]$. Item \eqref{item:ell-2d-3} is proved. 

Next, we will show \eqref{item:ell-2d-4}. Indeed, we can write, for $y\neq 0$,
\[
\ell_2(y)-\ell_3(y)=R_0^{d-2-2\theta}a(|y|)+A(R_0,R),
\] 
where $A(R_0,R)$ is a constant, and $a(r)=\tfbr(r)-\gamma^{-2}C_{\gamma,R}R^{2-d}e^{-\gamma r^2/R^2}$. By Lemma~\ref{lem:tfab}\eqref{item:tfab_der},
\begin{align*}
a'(r)&\le -cr^{1-d}+C\gamma^{-1}C_\gamma R^{-d}e^{-\gamma r^2/R^2}\\
&\le 
cr^{1-d}\left(-1+C\gamma^{-1}(\tfrac{r}{R})^d\right)<0
\end{align*}
for $r\in[0.5R-2, R]$ if $\gamma$ is chosen to be sufficiently large.
Hence, $\ell_2-\ell_3$ is radially decreasing in $\B_R\setminus\B_{0.5R-2}$. Item \eqref{item:ell-2d-4} then follows from the fact that $\ell_2=\ell_3$ on $\partial\B_{R/2}$. Item \eqref{item:ell-2d-5} is a consequence of \eqref{eq:L-e^x2/R2} in Lemma~\ref{lem:test-exponential}.
\end{proof}

\begin{proof}
[Proof of the lower bound in Theorem~\ref{thm:Green-bound}:]
It suffices to show that, for $x\in B_R$, 
\begin{equation}\label{eq:modified-lb-2d}
G_R^\omega(x)\gtrsim
R_0^{d-2-2\theta}(U(|x|+1)-U(R+1)).
\end{equation}
Recall $U(r)$ in \eqref{eq:def_ur}.
Indeed, \eqref{eq:modified-lb-2d} is equivalent to the lower bound of Theorem~\ref{thm:Green-bound} when $x\in B_{R-1}$. When $x\in B_R\setminus B_{R-1}$, $R\ge 2$, taking $y\in B_R$ with $|y|\le |x|-1$ and $|x-y|_1\le 2d$, by the assumption (A2), inequality \eqref{eq:modified-lb-2d} yields
\begin{align*}
G_R(x)
&\ge P_x^\omega(X_{|x-y|_1}=y)G_R(y)\\
&\gtrsim 
R_0^{d-2-2\theta}(U(|x|)-U(R+1))\\
&\gtrsim
R_0^{d-2-2\theta}(U(|x|+1)-U(R+2)).
\end{align*}

Our proof of \eqref{eq:modified-lb-2d} consists of several steps.
Let $C_0$ be as in Lemma~\ref{lem:eta} and recall $\tau=\tau_R$ in \eqref{def:tau}.

When $R\in(1, 2C_0\theta^2)$, by \eqref{eq:L-e^x2}, taking $A=A(\kappa)\ge 1$ sufficiently large, 
\[
L_\omega[G_R-(e^{-A|x|^2}-e^{-AR^2})]\le 0 \quad \text{ in }B_R.
\]
Since $G_R=0\ge (e^{-A|x|^2}-e^{-AR^2})$ on $\partial B_R$, by the maximum principle, we have $G_R\ge e^{-A|x|^2}-e^{-AR^2}$ in $B_R$. Thus, using the inequality $e^a\ge 1+a$ for $a\ge 0$, we get, for $x\in B_R$, 
(Note that $R\asymp 1$ in this case.)
\begin{align*}
G_R\ge e^{-AR^2}\left(e^{A(R^2-|x|^2)}-1\right)
\ge e^{-AR^2}(R^2-|x|^2)
\gtrsim R-|x|
\end{align*}	
By the fact $a\ge\log(1+a), a\ge 0$, we have $R-|x|\gtrsim \log\frac{R+1}{|x|+1}$. Moreover, since $R\asymp 1$, for $d\ge 3$, we also have $R-|x|\gtrsim (|x|+1)^{2-d}-(R+1)^{2-d}$. Thus \eqref{eq:modified-lb-2d} holds for this case.

 When $R\ge 2C_0\theta^2$, by 
assumption (A2), for $x\in B_R$ and any $y\in B_{C_0\theta^2}$,
\begin{align*}
G_R(x)
&\ge \sum_{i=0}^\infty P^x_\omega(X_i=y, i<\tau_R)P_\omega^y(X_{|y|_1}=0)\\
&\ge G_R(x,y)\kappa^{|y|_1}\gtrsim G_R(x,y),
\end{align*}	
and so (Recall $G_R(\cdot,\cdot)$ in Definition~\ref{def:green}.)
\begin{equation}\label{eq:compare-GR-H}
G_R(x)\gtrsim G_R(x,B_{C_0\theta^2})=:H_R(x), \quad x\in  B_R.
\end{equation}
Thus it suffices to obtain the corresponding lower bound for $H_R$ defined above.

When $R\in(2C_0\theta^2,4R_0]$, since (by Lemma~\ref{lem:eta}) $L_\omega(H_R-\tfe)\le 0$ in $B_R$, by the maximum principle, we have  $H_R\ge \tfe-\tfer(R)$ in $B_R$. Notice that
\begin{equation}
\label{eq:20200206-3}
\tfer(r_1)-\tfer(r_2)
\gtrsim
R_0^{d-2-2\theta}(U(r_1+1)-U(r_2+1)),\quad \forall r_1<r_2\le 4R_0.
\end{equation}
Indeed, for $d=2$, 
\begin{align*}
\tfer(r_1)-\tfer(r_2)
&\ge 
\left[\left(\tfrac{1+r_1^2}{1+r_2^2}\right)^{-\theta}-1\right](1+r_2^2)^{-\theta}\\
&\ge 
CR_0^{-2\theta}\log \tfrac{1+r_2^2}{1+r_1^2}
\ge 
CR_0^{-2\theta}\log \tfrac{1+r_2}{1+r_1},
\end{align*}
where we used the fact $a\ge \log(1+a)$ for $a\ge 0$ in the second inequality.
For $d\ge 3$, recalling that $\theta\ge d/2$ in \eqref{put:theta},
\begin{align*}
\tfer(r_1)-\tfer(r_2)&=(1+r_1^2)^{-\theta}-(1+r_2^2)^{-\theta}\\
&\ge (1+r_2^2)^{0.5d-1-\theta}[(1+r_1^2)^{1-0.5d}-(1+r_2^2)^{1-0.5d}]\\
&\gtrsim
R_0^{d-2-2\theta}[(1+r_1)^{2-d}-(1+r_2)^{2-d}]\
\end{align*}
Hence, we obtain $H_R\gtrsim R_0^{d-2-2\theta}(U(|x|+1)-U(R+1))$ for this case.

It remains to consider the case $R\ge 4R_0$. To this end, we will prove 
\begin{equation}
\label{eq:HR-ell}
(G_R\by{\eqref{eq:compare-GR-H}}{\gtrsim})H_R\ge \ell \quad \text{ in } \bar B_R.
\end{equation}
Assume by contradiction that \eqref{eq:HR-ell} fails, i.e., $\max_{\bar B_R}(\ell-H_R)>0$. 
By Lemma~\ref{lem:ell_properties}\eqref{item:ell-2d-1}, $\max_{\bar B_R}(\ell-H_R)$ is achieved outside of $B_{R_0}$. Further, note that $(\ell-H_R)|_{\partial B_R}\le 0$. By Lemma~\ref{lem:ell_properties} \eqref{item:ell-2d-5}, \eqref{item:ell-2d-4}, and the maximum principle, 
\[
\max_{\bar B_R\setminus B_{R/2}}(\ell-H_R)
\le 
\max_{\partial(B_R\setminus B_{R/2})}(\ell_3-H_R)
\le 
0\vee\max_{B_{R/2}\setminus B_{R_0}}(\ell-H_R).
\]
Hence, if $\max_{\bar B_R}(\ell-H_R)>0$, then there exists $x_0\in B_{R/2}\setminus B_{R_0}$ such that
\[
(\ell-H_R)(x_0)=\max_{\bar B_R}(\ell-H_R).
\]
Since $x_0\in B_{R/2}\setminus B_{R_0}$, by Lemma~\ref{lem:ell_properties}\eqref{item:ell-2d-3}, \eqref{item:ell-2d-4}, 
\[
(\ell_2-H_R)(x_0)\ge \max_{\bar B_{|x_0|/2}(x_0)}(\ell_2-H_R),
\]
which is equivalent to 
\[
(R_0^{d-2-2\theta}\tfb-H_R)(x_0)
\ge 
\max_{\bar B_{|x_0|/2}(x_0)}(R_0^{d-2-2\theta}\tfb-H_R).
\]
Without loss of generality, assume $\bar a=I$, and set, for $y\in\R^d\setminus\{0\}$,
\[
\tilde{\tfb}(y):=\tfb(y)-c|x_0|^{-(2+\delta)}\abs{\tfb(x_0)}|y-x_0|^2,
\]
where $c>0$ is chosen so that (Lemma~\ref{lem:tfab}\eqref{item:tfab_laplace}) $-\Delta\tilde{\tfb}\le 0$ for $|y|\ge C$. Then
\begin{equation}
\label{eq:20200206-1}
(R_0^{d-2-2\theta}\tilde\tfb-H_R)(x_0)
\ge 
\max_{\partial B_{|x_0|/2}(x_0)}(R_0^{d-2-2\theta}\tilde\tfb-H_R)+cR_0^{d-2-2\theta}|x_0|^{-\delta}\abs{\tfb(x_0)}.
\end{equation}

Next, let $g(x)=\tilde{\tfb}(x_0+\tfrac{|x_0|}{2}x)$, 
and let $v$ be the solution of 
\[
\left\{
\begin{array}{lr}
L_\omega v(x)=0 & x\in B_{|x_0|/2}(x_0)\\
v(x)=R_0^{d-2-2\theta}g(\tfrac{x-x_0}{|x-x_0|}) & x\in\partial B_{|x_0|/2}(x_0).
\end{array}
\right.
\]
Note that $v$ is close to $R_0^{d-2-2\theta}\tilde\tfb$ on $\partial B_{|x_0|/2}(x_0)$ in the sense that $\Abs{g(\tfrac{x-x_0}{|x-x_0|})-\tilde{\tfb}(x)}\le C|x_0|^{-1}\abs{\tfb(x_0)}$ for $x\in\partial B_{|x_0|/2}(x_0)$. 
Comparing the $L_\omega$-harmonic functions $v$ and $H_R$ in $\bar B_{|x_0|/2}(x_0)$ via the maximum principle, we have, for $x\in B_{|x_0|/2}(x_0)$,
\begin{equation}
\label{eq:ell-G-upper}
H_R(x)+\max_{\partial B_{|x_0|/2}(x_0)}(R_0^{d-2-2\theta}\tilde \tfb-H_R)\ge v(x)-cR_0^{d-2-2\theta}|x_0|^{-1}\abs{\tfb(x_0)}.
\end{equation}

By Proposition~\ref{thm:quant-homo}, for $x\in B_{|x_0|/2}(x_0)$,
\begin{equation}\label{eq:v-ell-upper}
v(x)
\ge 
\bar v(\tfrac{x-x_0}{|x_0|/2})-CR_0^{d-2-2\theta}|x_0|^{-\beta}[g]_{C^{0,1}(\partial\B_1)}\ge \bar v(\tfrac{x-x_0}{|x_0|/2})-CR_0^{d-2-2\theta}|x_0|^{-\beta}\abs{\tfb(x_0)},
\end{equation}
where $\bar v$ solves
\[
\left\{
\begin{array}{lr}
\tr(\bar aD^2\bar v)=\Delta \bar v=0 & x\in \B_1\\
\bar v(x)=R_0^{d-2-2\theta}g(x) & x\in\partial\B_1.
\end{array}
\right.
\]
Furthermore, using the fact that $\Delta g(x)=|x_0|^2\Delta\tilde \tfb(x_0+\tfrac{|x_0|}{2}x)\ge 0$ for $x\in\B_1$, we get $\bar v\ge R_0^{d-2-2\theta}g$ in $\B_1$. Therefore, by \eqref{eq:ell-G-upper}, \eqref{eq:v-ell-upper}, for $x\in B_{|x_0|/2}(x_0)$,
\[
R_0^{d-2-2\theta}g(\tfrac{x-x_0}{|x_0|/2})-H_R(x)
\le 
\max_{\partial B_{|x_0|/2}(x_0)}(R_0^{d-2-2\theta}\tilde\tfb-H_R)+CR_0^{d-2-2\theta}|x_0|^{-\beta}\abs{\tfb(x_0)}.
\]
Noting that $g(\tfrac{x-x_0}{|x_0|/2})=\tilde \tfb(x)$, the above inequality contradicts \eqref{eq:20200206-1} since $|x_0|> R_0\ge K$, where $K$ is chosen to be sufficiently large.
Inequality \eqref{eq:HR-ell} is proved.

When $x\in B_{R/2}\setminus B_{R_0}$, we have $G_R\by{\eqref{eq:HR-ell}}\gtrsim \ell_2$, and, 
by Lemma~\ref{lem:tfab}\eqref{item:tfab_der},
\begin{align}
\label{eq:20200206-2}
\ell_2&\ge
R_0^{d-2-2\theta}[\tfbr(|x|)-\tfbr(R)]\nn\\
&\gtrsim
R_0^{d-2-2\theta}
\int_{|x|}^R r^{1-d}\dd r \nn\\&
\gtrsim
R_0^{d-2-2\theta}(U(|x|+1)-U(R+1)).
\end{align}
When $x\in B_{R_0}$,
\begin{align*}
G_R&\by{\eqref{eq:HR-ell}}\gtrsim \ell_1= E_\omega^x[\ell_2(X_{s_0})+\eta(x)-\eta(X_{s_0})]\\
&\stackrel{\eqref{eq:20200206-2},\eqref{eq:20200206-3}}{\gtrsim }
R_0^{d-2-2\theta}E_\omega^x\left[
U(|X_{s_0}|+1)-U(R+1)+U(|x|+1)-U(|X_{s_0}|+1)
\right]\\
&=
R_0^{d-2-2\theta}[U(|x|+1)-U(R+1)].
\end{align*}	
Finally, when $x\in B_R\setminus B_{R/2}$, using the inequality  $e^a\ge 1+a$ for $a\ge 0$,
\begin{align*}
G_R\by{\eqref{eq:HR-ell}}\gtrsim 
\ell_3
&\gtrsim  R_0^{d-2-2\theta}R^{2-d}\left(e^{\gamma(1-|x|^2/R^2)}-1\right)\\
&\gtrsim R_0^{d-2-2\theta}R^{2-d}(1-\tfrac{|x|}{R}).
\end{align*}	
For $d=2$, note that
$1-\tfrac{|x|}{R}\asymp \tfrac{R}{|x|}-1\ge \log\frac{R}{|x|}$.
For $d=3$, clearly 
\[
R^{2-d}(1-\tfrac{|x|}{R})\gtrsim |x|^{2-d}(1-(\tfrac{|x|}{R})^{d-2}).
\]
Our proof is complete.
\end{proof}

%

\begin{lemma}\label{lem:potential-lim-bd}
Assume {\rm(A1), (A2)}. When $d=2$, then for $\mb P$-almost all $\omega$,
\[
c\rho(\omega)\le \varliminf_{|x|\to\infty}\frac{A^\omega(x)}{\log|x|}
\le \varlimsup_{|x|\to\infty}\frac{A^\omega(x)}{\log|x|}
\le C\rho(\omega).
\]
\end{lemma}

\section{Heat kernel bounds and consequences}

\subsection{Integrability of $\rho$ and the heat kernel bounds}

Using bounds of the Green functions,  we will obtain the exponential integrability (under $\mb P$) of the Radon-Nikodym derivative $\rho(\omega)$ (defined in \eqref{eq:def-rho})
and the heat kernel of the RWRE.

The goal of this section is to prove Theorem~\ref{thm:hk-bounds}. Recall the continuous-time RWRE in Definition~\ref{def:rwre-continuous} and its transition kernel $p_t(x,y)$.  We remark that for the time continuous random walk $(Y_t)$, setting 
\begin{equation}\label{eq:def-tau-y}
\tau^Y=\tau_R^Y:=\inf\{t\ge 0: Y_t\notin B_R\},
\end{equation}
the corresponding Green functions of $(Y_t)$ can be defined similarly as
\[
\int_0^\infty p_t^\omega(x,S)\dd t,\, d\ge 3,
\quad \text{ and }\quad 
E_\omega^x\left[\int_0^{\tau^Y}\mathbbm{1}_{Y_t\in S}\dd t\right], \, d\ge 2,
\]
and they have the same values as $G(x,S)$ and $G_R(x)$, respectively.  Thus we do not need to distinguish notations in discrete and continuous time cases and use $G(x,S)$ and $G_R(x)$ to denote Green's functions in both settings.

\begin{corollary}\label{cor:green-time}
Assume {\rm(A1), (A2)} and $d=2$. 
For any $\error>0$, there exists a random variable $\tx=\tx(\omega,d,\kappa,\error)>0$ with 
$\mb E[\exp(c\tx^{d-\error})]<\infty$ such that  $\mb P$-almost surely, for $R>0$,
\[
\int_0^{R^2}p_t^\omega(x,0)\dd t
\lesssim
\tx
(1+\log\tfrac{R+1}{|x|+1}),
\quad \forall x\in B_R.
\]
\end{corollary}
\begin{proof}
We only consider $R\ge 4$. Let $R_k:=2^{k-1}R$, and define recursively $T_0=0$,
\[
T_k:=\min\{t\ge T_{k-1}:Y_t\notin B_{R_k}\}, \quad k\in\N.
\]
Set $N_R=\max\{n\ge 0:T_n\le R^2\}$. Then, using the strong Markov property, a.s.,
\begin{align*}
E_\omega^x\left[\int_0^{R^2}\mathbbm{1}_{Y_t=0}\dd t\right]
&\le 
E_\omega^x\left[
\sum_{k=1}^\infty\int_{T_{k-1}}^{T_k}\mathbbm{1}_{\{Y_t=0,T_{k-1}\le R^2\}}\dd t
\right]\\
&\le 
\sum_{k=1}^\infty
E_\omega^x\left[
G_{R_k}^\omega(Y_{T_{k-1}})
\mathbbm{1}_{\{T_{k-1}\le R^2\}}\right]\\
&\lesssim 
\tx\left(G_R^\omega(x)+\sum_{k=2}^\infty P_\omega^x(T_{k-1}\le R^2)\right),
\end{align*}	
 where we used (by Theorem~\ref{thm:Green-bound}) the fact $G_{R_{k}}^\omega(Y_{T_{k-1}})\lesssim \tx$, $k\ge 2$, in the last inequality. Note that $(Y_t)$ is a martingale. By Hoeffding's inequality, for $k\ge 1$,
\begin{align*}
P_\omega^x(T_k\le R^2)
\le 
P_\omega^x \left(\sup_{t\le R^2}|Y_t|\ge R_k\right)\le 
C e^{-cR_k^2/R^2}\le C\exp(-c4^k).
\end{align*}	
The conclusion follows.
\end{proof}

\begin{proof}
[Proof of Theorem~\ref{thm:hk-bounds}:]
First, we will show the upper bounds in \eqref{item:rho-bounds} and \eqref{item:rho-hke-expmmt}. To this end, for $r\ge 1$, we take $x_0\in \partial B_r$. 
We claim that
\begin{equation}\label{eq:220922}
|x_0|^{2-d}\tx^{d-1}
\gtrsim 
\int_{r^2/2}^{r^2}
p^\omega_t(x_0,0)\dd t
\gtrsim r^2\frac{\rho_\omega(0)}{\rho_\omega(B_r)}. 
\end{equation}
Indeed,  the lower bound of \eqref{eq:220922} follows from integrating the lower bound of \eqref{eq:hke-random}.  For $d=2$,  the upper bound in \eqref{eq:220922} is a consequence of Corollary~\ref{cor:green-time}.  When $d\ge 3$, the upper bound of \eqref{eq:220922} follows from Corollary~\ref{cor:a_and_g}.
Note that $|x_0|\asymp r$. The upper bound in \eqref{item:rho-hke-expmmt} is proved. The upper bound in \eqref{item:rho-bounds} then follows from taking $r\to\infty$ and the ergodic theorem.

To obtain the lower bound in \eqref{item:rho-hke-expmmt}, for $r\ge 5$. Recall $\tau_r$ in \eqref{def:tau} and $G_r(\cdot,\cdot)$ in Definition~\ref{def:green}. 
For any fixed $y_0 \in \partial B_{r/2}$, the function $v(x)=G_r(y_0,x)/\rho_\omega(x)$ solves the {\it adjoint equation}
\begin{equation}\label{eq:adj-eq}
L_\omega^* v(x):=\sum_{y}\omega^*(x,y)[v(y)-v(x)]=0, \quad x\in B_{r/2},
\end{equation}
where 
\begin{equation}\label{eq:def-adjoint}
\omega^*(x,y):=\rho_\omega(y)\omega(y,x)/\rho_\omega(x).
\end{equation}
Here, we used the facts that $\sum_y \rho_\omega(y) \omega(y,x)=\rho_\omega(x)$ and $\sum_y G_r(y_0,y) \omega(y,x) = G_r(y_0,x)$.
 By the Harnack inequality for the adjoint operator \cite[Theorem 6]{DG-19}, we have $v(0)\asymp  v(x)$ for all $x\in B_{r/4}$. Hence
\begin{equation}\label{eq:G-ball-rho-compare}
G_r(y_0, 0)\frac{\rho_\omega(B_{r/4})}{\rho_\omega(0)}
\asymp
G_r(y_0, B_{r/4}).
\end{equation}
Moreover,  since $(|X_n|^2-n)$ is a martingale under $P_\omega$, by the optional stopping lemma we get $E_\omega^{y_0}[|X_{\tau_r}|^2-\tau_r]=|y_0|^2\ge 0$, and so
\[
G_r(y_0, B_{r/4})\le E_\omega^{y_0}[\tau_r]\le E_\omega^{y_0}[|X_{\tau_r}|^2]\le Cr^2.
\]
The above inequality, together with \eqref{eq:G-ball-rho-compare} and  Theorem~\ref{thm:Green-bound}, yields
\begin{align*}
\frac{\rho_\omega(0)}{\rho_\omega(B_{r/4})}
\gtrsim
\frac{G_r(y_0, 0)}{G_r(y_0, B_{r/4})}
\gtrsim 
\frac{\tx^{-s}r^{2-d}}{r^2}
\gtrsim
\tx^{-s}r^{-d}.
\end{align*}
The lower bound in Theorem~\ref{thm:hk-bounds}\eqref{item:rho-hke-expmmt} follows. Letting $r\to\infty$, we also get the lower bound in \eqref{item:rho-bounds}.
\end{proof}


\subsection{Optimal semigroup decay for $d\ge 3$}

The Efron-Stein inequality \eqref{eq:bblm} of Boucheron, Bousquet, and Massart \cite{BBLM-05} will be used in our derivation of the variance decay for the semi-group.

Let $\omega'(x), x\in\Z^d$, be independent copies of $\omega(x), x\in\Z^d$. For any $y\in\Z^d$, let $\omega'_y\in\Omega$ be the environment such that
\[
\omega'_y(x)=
\left\{
\begin{array}{lr}
&\omega(x) \quad \text{ if }x\neq y,\\
&\omega'(y) \quad \text{ if }x=y.
\end{array}
\right.
\]
That is, $\omega'_y$ is a modification of $\omega$ only at location $y$. 
For any measurable function $Z$ of the environment $\omega$, we write, for $y\in\Z^d$,
\begin{equation}\label{eq:def_vert_der}
Z_y'=Z(\omega_y'), \quad
\vd{y} Z(\omega)=Z'_y-Z, 
\end{equation}
and set
\begin{equation}\label{def:v}
V(Z)=\sum_{y\in\Z^d}(\vd{y}Z)^2. 
\end{equation}

By an $L_p$ version of Efron-Stein inequality \cite[Theorem~3]{BBLM-05}, for $q\ge 2$,
\begin{equation}\label{eq:bblm}
\mb E[|Z-\mb EZ|^q]\le Cq^{q/2}\mb E[V^{q/2}].
\end{equation}

Following the strategy of \cite{GNO-15}, our proof of the diffusive decay of the semi-group $\{P_t\}$ will make use of the Efron-Stein type inequality \eqref{eq:bblm} and the Duhamel representation formula \eqref{eq:vd_duhamel} for the vertical derivative. 
Let us reemphasize that, in the non-divergence form setting, there is no deterministic Gaussian bounds for the heat kernel, and the steady state $\Q$ of the environment process $(\evp t)_{t\ge 0}$ is not the same as the original measure $\mb P$.
To overcome these difficulties, we employ crucially the heat kernel estimates and the (negative and positive) moment bounds of the Radon-Nikodym derivative $\tfrac{\dd \mb P}{\dd\Q}$ in Theorem~\ref{thm:hk-bounds}.

For any $\zeta\in L^1(\Omega)$, we write
\[
v(t):=P_t\zeta(\omega).
\]
Then, its stationary extension $\bar v(t,x)$ solves the parabolic equation
\begin{equation}\label{eq:heat_eq_v}
\left\{
\begin{array}{lr}
\partial_t\bar v(t,x)-L_\omega\bar v(t,x)=g(t,x) & t\ge 0, x\in\Z^d,\\
v(0,x)=g_0(x) &x\in\Z^d,
\end{array}
\right.
\end{equation}
with $g(t,x)=0$ and $g_0(x)=\bar \zeta(x;\omega)$. 
In general, the solution of \eqref{eq:heat_eq_v} can be represented by Duhamel's formula
\begin{equation}\label{eq:duhamel}
\bar{v}(t,x)=\sum_z p_t^\omega(x,z)g_0(z)+\int_0^t p^\omega_{t-s}(x,z)g(s,z)\dd s.
\end{equation}

To apply \eqref{eq:bblm}, recall notations $Z_y'$ and $\vd y Z$ in \eqref{eq:def_vert_der}. By enlarging the probability space, we still use $\mb P$ to denote the joint law of $(\omega,\omega')$. 
For $y\in\Z^d$, applying $\vd y$ to \eqref{eq:heat_eq_v}, we get that $\vd y\bar v$ satisfies
\[
\left\{
\begin{array}{lr}
\partial_t(\vd y\bar v)(t,x)-L_\omega(\vd y v)(t,x)=(\vd y \omega_i(x))\nabla_i^2\bar v_y'(t,x) & t\ge 0, x\in\Z^d,\\
(\vd y\bar v)(0,x)=\vd y\bar \zeta(x), &x\in\Z^d.
\end{array}
\right.
\]
Here we used the convention of summation over repeated integer indices.
Hence, by formula \eqref{eq:duhamel}, $\vd y\bar v$ has the representation
\begin{align}\label{eq:vd_duhamel}
\vd y\bar v(t,x)
&=
\sum_z\left[ p_t^\omega(x,z)\vd y\bar\zeta(z)+\sum_{i=1}^d\int_0^t p^\omega_{t-s}(x,z)
(\vd y \omega_i(z))\nabla_i^2\bar v_y'(s,z)\dd s\right]\nn\\
&=
\sum_{z\in y+\supp(\zeta)} p_t^\omega(x,z)\vd y\bar\zeta(z)+\sum_{i=1}^d\int_0^t p^\omega_{t-s}(x,y)
(\vd y \omega_i(y))\nabla_i^2\bar v_y'(s,y)\dd s,
\end{align}
where in the last equality we used the fact $\vd y\omega_i(z)=0$ for $y\neq z$, and that $\vd y\bar\zeta(z)=0$ for $z\notin y+\supp(\zeta)$.

\begin{proof}
[Proof of Theorem~\ref{thm:var_decay}:]

Recall the notation $V(\cdot)$ in \eqref{def:v}.  We write 
\[
K(y,s)=K(y,s;\omega,\omega'):=\sum_{i=1}^d(\vd y \omega_i(y))\nabla_i^2\bar v_y'(s,y).
\]
We also write $\tx_z:=\tx(\theta_z\omega)$ and $S:=\supp(\zeta)$. 
Without loss of generality, assume $E_\Q\zeta=0$.
Using \eqref{eq:vd_duhamel}, for any $p>1$,
\begin{align}\label{eq:var_decay_2parts}
\norm{V(v(t))}_p
&=
\norm{\sum_y(\vd y\bar v(t,0))^2}_p\nn\\
&\lesssim
\norm{\sum_y
\left(
\sum_{z\in S} p_t^\omega(0,z+y)
\right)^2
}_p
+
\norm{
\sum_y\left(
\int_0^t p^\omega_{t-s}(0,y)K(y,s)\dd s
\right)^2
}_p\nn\\
&\lesssim
(\#S)^2  \norm{\sum_y p_t^\omega(0,y)^2}_{p}
+
\int_0^t
\norm{
\sum_y
p_{t-s}^\omega(0,y)^2K(y,s)^2
}_p^{1/2}
\dd s\nn\\
&=:(\#S)^2  \rn{I}+\rn{II}^2,
\end{align}	
where $\#S$ denote the cardinality of $S$, and in the last inequality we applied the Cauchy-Schwarz inequality and  Minkowski's integral inequality to the two norms respectively.
Then, by Theorem~\ref{thm:hk-bounds}\eqref{item:hke-expmmt}, 
\begin{align}\label{eq:partI_var_decay}
 \rn{I}
=
 \norm{\sum_y p_t^\omega(0,y)^2}_{p}
 &\lesssim (1+t)^{-d}\norm{\sum_z\tx_z^{2(d-1)} e^{-c\mf h(|z|,t)}}_{p}\nn\\
 &\lesssim
 (1+t)^{-d}\sum_z e^{-c\mf h(|z|,t)}\norm{\tx_z^{2(d-1)}}_{p}\nn\\
 &\lesssim
 (1+t)^{-d/2}\norm{\tx^{2(d-1)p}}_{1}^{1/p},
 \end{align}	 
 where in the last inequality we used the translation invariance of $\mb P$. 
 
 Further,  using Theorem~\ref{thm:hk-bounds}\eqref{item:hke-expmmt} again,
\begin{align}\label{eq:partII_var_decay}
\rn{II}
&\lesssim
\int_0^t
(1+t-s)^{-d/2}\left(
\sum_y e^{-c\mf h(|y|,t-s)}
\norm{\ms H_y^{2(d-1)}
K(y,s)^2}_p
\right)^{1/2}
\dd s\nn\\
&\lesssim
\int_0^t(1+t-s)^{-d/4}
\norm{
\ms H^{2(d-1)p}K(0,s)^2
}_1^{1/(2p)}\dd s
\end{align}	
where  in the last inequality we used the translation-invariance of $\mb P$ and the fact that $\norm{\bar v}_\infty\lesssim 1$.

Let $p'>1$ denote the H\"older conjugate of $p>1$. Since 
\begin{align}\label{eq:jun12020-2}
\norm{
\ms H^{2(d-1)p}K(0,s)^2
}_1
&\le 
\norm{\ms H^{2(d-1)p}}_{p'}\norm{K(0,s)^2}_p\nn\\
&=C_p \norm{K(0,s)^2}_p\nn\\
&\le
C_p 
\sum_{i=1}^d\norm{\nabla_i^2\bar v(s,0)}_{2p}^2
\nn\\
&\le
 C_p
 \sum_{e:|e|=1}\norm{|\bar v(s,e)-\bar v(s,0)|^{2/p}}_1^{1/p}\nn\\
&\by{\text{H\"older}}\le
C_p\norm{\rho_\omega^{-1/p}}_{p'}^{1/p}
\sum_{e:|e|=1}\norm{\rho_\omega |\bar v(s,e)-\bar v(s,0)|^2}_1^{1/p^2}
\end{align}	
where in the third inequality we used the fact $\norm{\bar v}_\infty\le 1$. Then, setting
\[
u(t):=\var_\Q(v(t)),
\]
by \eqref{eq:jun12020-2}, Theorem~\ref{thm:hk-bounds}\eqref{item:rho-hke-expmmt} and Lemma~\ref{lem:caccip}, we obtain
\begin{equation}
\label{eq:jun12020-1}
\norm{
\ms H^{2(d-1)p}K(0,s)^2
}_1\le 
C_p\sum_e E_\Q[(\bar v(s,e)-\bar v(s,0))^2]^{1/p^2}
\le 
C_p(-\frac{\dd}{\dd t}u)^{1/p^2}.
\end{equation}
This inequality, together with \eqref{eq:var_decay_2parts},\eqref{eq:partI_var_decay}, \eqref{eq:partII_var_decay}, implies
\begin{align*}
\norm{V(v(t))}_p^{1/2}
\lesssim_p
(1+t)^{-d/4}+\int_0^t(1+t-s)^{-d/4}(-\frac{\dd}{\dd t}u)^{1/(2p^3)}\dd s,
\end{align*}	
where $\lesssim_p$ means that the multiplicative constant depends on $(p,d,\kappa)$.

Furthermore, by H\"older's inequality and Theorem~\ref{thm:hk-bounds}\eqref{item:rho-bounds}, 
\[
u(t)\le E_\Q[(v(t)-\mb Ev(t))^2]
\le \norm{\rho_\omega}_{p'}\norm{(v(t)-\mb Ev(t))^2}_p
\lesssim_p
\norm{V(v(t))}_p
\]
where we applied \eqref{eq:bblm} in the last inequality.

Therefore, we conclude that, for any $p>1$, 
\[
u(t)^{1/2}\lesssim_p
(1+t)^{-d/4}+\int_0^t(1+t-s)^{-d/4}(-\tfrac{\dd}{\dd t}u(s))^{1/(2p^3)}\dd s.
\]
When $d\ge 3$, we can take $p>1$ sufficiently close to $1$ and apply   \cite[Lemma 3.5]{AN-17}) 
(Under the notation of \cite{AN-17}, we apply it to the case $\gamma=d/4$ and $\delta=1/p^3$.)
to obtain
\[
\var_\Q(v(t))=u(t)\lesssim (1+t)^{-d/2}.
\]
Thus, \eqref{eq:var_decay_q} is proved. By H\"older's inequality,
\[
\norm{v(t)}_1\le \norm{\rho_\omega^{-1}}_1^{1/2}\norm{\rho_\omega v(t)^2}_1^{1/2}\le Cu(t)^{1/2}\le C(1+t)^{-d/4}.
\]
Then, by the triangle inequality, we get $E_\Q[(v(t)-\mb Ev(t))^2]\lesssim (1+t)^{-d/2}$. By H\"older's inequality, for any $q>1$,
\[
\norm{\abs{v(t)-\mb Ev(t)}^{2/q}}_{1}
\le \norm{\rho_\omega^{-1/q}}_{q'}\norm{\rho^{1/q}\abs{v(t)-\mb Ev(t)}^{2/q}}_{q}
\lesssim_q
E_\Q[(v(t)-\mb Ev(t))^2]^{1/q}.
\]
Display \eqref{eq:var_decay_p} is proved.
\end{proof}

\begin{lemma}
\label{lem:caccip}
For any bounded measurable function $\zeta\in\R^{\Omega}$, 
\[
\frac{\dd}{\dd t}\var_{\Q}(v(t))\lesssim
-\sum_{e:|e|=1}E_\Q\left[(\bar v(t,e)-\bar v(t,0))^2\right],
\]
where $L_\omega$ only acts on the spatial variable.
\end{lemma}
\begin{proof}
Without loss of generality, assume $\mb E[v(t)]=0$. Then 
\begin{align*}
\frac{\dd}{\dd t}E_\Q[v(t)^2]
&=2 E_\Q[v(t) L_\omega\bar v(t,0)]\\
&=
E_\Q\left[
L_\omega(\bar v(t,0)^2)-\sum_{e:|e|=1}\omega(0,e)(\bar v(t,e)-\bar v(t,0))^2
\right]\\
&=
-\sum_{e:|e|=1}E_\Q\left[\omega(0,e)(\bar v(t,e)-\bar v(t,0))^2
\right]
\end{align*}	
where in the last equality we used the fact that (since $(\evp t)_{t\ge 0}$ is stationary under $\Q\times P_\omega$) for any $f\in L^1(\Q)$, $E_\Q[L_\omega\bar f(0;\omega)]=0$. The lemma follows by the uniform ellipticity assumption.
\end{proof}

\subsection{Proof of Corollary~\ref{cor:ex_sta_cor}:  Existence of a stationary corrector for $d\ge 5$}

\begin{proof}[Proof of Corollary~\ref{cor:ex_sta_cor}:]
 Without loss of generality, assume $E_\Q[\zeta]=0$. By Theorem~\ref{thm:var_decay}, 
since $\int_0^\infty (1+t)^{-d/4}\dd t<\infty$ when $d\ge 5$,  
 the limit 
 \begin{equation}\label{eq:def_stat_cor}
\phi(\omega)=\lim_{t\to\infty}\phi_t(\omega):=-\lim_{t\to\infty}\int_0^t (P_s\zeta) \dd s
 \end{equation}
exists in $L^p(\mb P)$ form any $p\in(0,2)$. By \eqref{eq:heat_eq_v},  $\bar\phi_t(x)$ satisfies
\[
L_\omega \bar\phi_t(x)=\zeta(\theta_x \omega)-P_t\zeta(\theta_x \omega), \quad x\in\Z^d.
\] 
Note that by Theorem~\ref{thm:var_decay}, $\lim_{t\to\infty}P_t\zeta=0$ in $L^p(\mb P)$, $\forall p\in(0,2)$. 
Therefore, taking $L^p(\mb P)$ limits as $t\to\infty$, we conclude that $\bar\phi$ satisfies $L_\omega\bar\phi(x)=\bar\zeta(x)$, $x\in\Z^d$.
 \end{proof}

\newpage
\newtheorem{atheorem}{Theorem}
\numberwithin{atheorem}{section}
\newtheorem{alemma}[atheorem]{Lemma}
\newtheorem{acorollary}[atheorem]{Corollary}

\appendix
\section{Appendix}
Define the parabolic operator $\ms L_\omega$ as
\[
\ms L_\omega u(x,t)=\sum_{y:y\sim x}\omega(x,y)[u(y,t)-u(x,t)]-\partial_t u(x,t)
\]
for every function $u:\Z^d\times\R\to\R$ which is differentiable in $t$. 
\begin{atheorem}(\cite[Proposition~5]{DG-19})
\label{thm:harnack-para}
Assume $\frac{\omega}{\tr\omega}>2\kappa I$ for some $\kappa>0$ and $R>0$. Any non-negative function $u$ with $\ms L_\omega u=0$ in $B_{2R}\times(0,4R^2)$ satisfies
\[
\sup_{B_{R}\times(R^2,2R^2)}u\le 
C\inf_{B_R\times(3R^2,4R^2)}u.
\]
\end{atheorem}

\begin{acorollary}\label{acor:Kry-Saf}
Assume $\frac{\omega}{\tr\omega}>2\kappa I$ for some $\kappa>0$ and $(x_0,t_0)\in\Z^d\times\R$.
There exists $\gamma=\gamma(d,\kappa)\in(0,1)$ such that any non-negative function $u$ with $\ms L_\omega u=0$ in $B_{R}(x_0)\times(t_0-R^2,t_0)$, $R>0$, satisfies
\[
|u(\hat x)-u(\hat y)|\le C\left(\frac{r}{R}\right)^\gamma\sup_{B_R(x_0)\times(t_0-R^2,t_0)}u
\]
for all $\hat x,\hat y\in B_r(x_0)\times(t_0-r^2,t_0)$ and $r\in(0,R)$.
\end{acorollary}

\subsection{Proof of Proposition~\ref{thm:quant-homo}}

Before giving a proof, recall that by \cite[Proposition~2.1]{GPT-19}, for any $p\in(0,d)$, there exists $\delta_p$ depending on  $(d,\kappa,p)$ such that for any $R>0$, the solution $\phi:\bar B_R\to\R$ of 
\begin{equation}
\label{eq:corrector}
\left\{
\begin{array}{lr}
L_\omega\phi=\bar a-a  &\text{ in }B_R,\\
\phi=0 &\text{ on }\partial B_R
\end{array}
\right.
\end{equation}
 satisfies 
\begin{equation}\label{eq:crtr-estimate}
\mb P(\max_{B_R}|\phi|\ge CR^{2-\delta_p})\le C\exp(-cR^p).
\end{equation} 

Set $\delta:=\delta_1$. For $q\in(0,d)$, let $\gamma=\gamma(d,\kappa,q)$ be the constant
\begin{equation}\label{eq:def-gamma}
\gamma=\min \left\{\frac{d-q}{d(1+\delta)},\frac{1}{2}\right\}
\end{equation}
and set
\[
R_0:=R^\gamma, \quad \sigma=\min\{n\ge 0:X_n-X_0\notin B_{R_0}\}.
\]
Let
\[
\omega_0=\tfrac{\omega}{\tr(\omega)}, \quad \psi_0=\psi_0(\omega)=\tfrac{\psi}{\tr(\omega)}.
\]
Following \cite[Definition~4.1]{GPT-19}, we define bad points.
\begin{definition}
Let $\delta=\delta(d,\kappa)$ be as above. We say that a point is good (and otherwise bad) if for any $\zeta(\omega)\in\{\psi_0,\omega_0\}$,
\[
\Abs{
E_\omega\big[\sum_{i=0}^{\sigma-1}(E_{\Q}[\zeta]-\zeta(\bar\omega_i))\big]
}\le 
C\norm{\zeta}_\infty R_0^{2-\delta}.
\]
Note that by \eqref{eq:crtr-estimate}, $\mb P(x\text{ is bad})\le Ce^{-cR_0}$.
\end{definition}

We will give first the proof for the special case $g\in C^{2,\alpha}(\partial\B_1)$. It is a small modification of the proof of \cite[Theorem~1.5]{GPT-19}.
\begin{proof}
[Proof of Proposition~\ref{thm:quant-homo} for the case $f\in C^\alpha(\mb B_1), g\in C^{2,\alpha}(\partial\B_1)$ and $y=0$:]
Note that if $g\in C^{2,\alpha}(\partial\B_1)$, it can be extended to be a function $\tilde g\in C^{2,\alpha}(\B_2)$ such that
\[
|\tilde g|_{2,\alpha;\B_2}
\le 
C\abs{g}_{2,\alpha;\partial\B_1}.
\]
By \cite[Theorem~6.6]{GiTr}, 
\begin{equation}\label{eq:regularity}
|\bar u|_{2,\alpha,\B_1}\lesssim |f|_{0,\alpha;\B_1}\norm{\tfrac{\psi}{\tr(\omega)}}_\infty+|g|_{2,\alpha;\partial\B_1}=:A
\end{equation}
\begin{enumerate}[Step 1.]
\item\label{item:step1} Set $\bar u_R(x)=\bar u(x/R)$ for $x\in \bar \B_R$.
 We will show that in $B_R$,  $\bar u_R$ is very close to the solution $\hat u:\bar B_R\to\R$ of 
\[
\left\{
\begin{array}{lr}
L_\omega\hat u=\frac{1}{2}\tr[\omega_0 D^2\bar u_R] &\text{ in }B_R\\
\hat u=g(\frac{x}{|x|}) &\text{ on }\partial B_R.
\end{array}
\right.
\]
To this end, define $u_+, u_-$ by $u_{\pm}(x)=\tilde g(\frac{x}{R+1})\pm CA\frac{(R+1)^2-|x|^2}{R^2}$, $x\in\bar\B_{R+1}$. 
Here $A$ is as defined in \eqref{eq:regularity}.
Then, for $x\in\B_{R+1}$, taking $C$ large enough,
\begin{align*}
\tr[\bar aD^2(u_+-\bar u_{R+1})]
&\le \frac{c}{R^2}(|g|_{2,\alpha;\B_1}+|f|_{0;\B_1}-CA)\le 0,
\end{align*}	
and similarly $\tr[\bar a(u_--\bar u_{R+1})]\ge 0$. The comparison principle then yields
\[
u_-\le \bar u_{R+1}\le u_+ \quad \text{ in }\B_{R+1}.
\]
In particular, for $x\in\partial B_R$, 
$
\abs{
\bar u_{R+1}(x)-\tilde{g}(\tfrac{x}{R+1})
}
\lesssim A\frac{(R+1)^2-|x|^2}{R^2}
\lesssim \frac{A}{R}$ and so 
\begin{equation}\label{eq:bdry-compare}
\max_{\partial B_R}|\hat u-\bar u_{R+1}|=\max_{x\in\partial B_R}|g(\tfrac x{|x|})-\bar u_{R+1}(x)|\lesssim \frac{A}{R}.
\end{equation}
Moreover, noting that $D^2\bar u_R(x)=R^{-2}D^2\bar u(\tfrac xR)$, in $B_R$,
\begin{align}\label{eq:laplacian-compare}
\abs{L_\omega(\hat u-\bar u_{R+1})}
&=\abs{\tr[\omega_0(D^2\bar u_R-\nabla^2\bar u_{R+1})]}\nn\\
&\le 
\abs{\tr[\omega_0(D^2\bar u_R-D^2\bar u_{R+1})]}
+\abs{\tr[\omega_0(D^2\bar u_{R+1}-\nabla^2\bar u_{R+1})]}\nn\\
&\lesssim
R^{-2-\alpha}|\bar u|_{2,\alpha;\bar\B_1}\stackrel{\eqref{eq:regularity}}{\lesssim} AR^{-2-\alpha}.
\end{align}	
Hence, by the ABP maximum principle \cite[Lemmas 2.3 and 2.4]{GPT-19}, 
\eqref{eq:bdry-compare} and \eqref{eq:laplacian-compare} imply
$\max_{B_R}|\hat u-\bar u_{R+1}|
\lesssim
AR^{-\alpha}$, and so, by \eqref{eq:regularity},
\[
\max_{B_R}|\hat u-\bar u_R|\lesssim AR^{-\alpha}.
\]

\item\label{item:step2}
Let $v=u-\hat u$. Then $v$ solves
\[
\left\{
\begin{array}{lr}
L_\omega v=\tfrac12\tr[(\bar a-\omega_0)D^2\bar u_R]
+\tfrac{1}{R^2}f(\tfrac{x}{R})(\psi_0-\bar\psi)
 &\text{ in }B_R\\
v=0 &\text{ on }\partial B_R.
\end{array}
\right.
\]
Note that, for $x\in B_{R-R_0}$ and $y\in B_{R_0}(x)$,
\begin{align*}
&|D^2\bar u_R(x)-D^2\bar u_R(y)|\le R^{-2}(\tfrac{R_0}{R})^\alpha[\bar u]_{2,\alpha;\B_1}
\stackrel{\eqref{eq:regularity}}{\lesssim}
AR^{-2}(\tfrac{R_0}{R})^\alpha,\\
&\abs{f(\tfrac{x}{R})-f(\tfrac{y}{R})}\le \left(\tfrac{R_0}{R}\right)^\alpha[f]_{\alpha;\B_1}.
\end{align*}	
Hence, if $x\in B_{R-R_0}$ is a good point, setting 
\[\bar\omega_0^i:=\omega_0(X_i) 
\quad\text{ and }\quad
\psi_0^i:=\psi_0(\theta_{X_i}\omega),
\]
and noting that $E_\omega^x[\sigma]\le (R_0+1)^2$, we have
\begin{align*}
&E_\omega^x[v(X_\sigma)-v(x)]\\
&=
E_\omega^x\left[
\sum_{i=0}^{\sigma-1}\tfrac{1}{2}\tr[(\bar\omega_0^i-\bar a)D^2\bar u_R(X_i)]+\tfrac{1}{R^2}f(\tfrac{X_i}{R})(\bar\psi-\psi_0^i)
\right]\\
&\lesssim
\tr\left[
E_\omega^x[
\sum_{i=0}^{\sigma-1}(\bar\omega_0^i-\bar a)]
D^2\bar u_R(x)\right]+\tfrac{1}{R^2}f(\tfrac{x}{R})E_\omega^x[\sum_{i=0}^{\sigma -1}(\bar\psi-\psi_0^i)]+
A\tfrac{R_0^\alpha}{R^{2+\alpha}} E_\omega^x[\sigma]\\
&\lesssim
\tfrac{1}{R^2}\Abs{E_\omega^x[
\sum_{i=0}^{\sigma-1}(\bar\omega_0^i-\bar a)]}|\bar u|_{2;\B_1}
+\tfrac{1}{R^2}|f|_{0;\B_1}\Abs{E_\omega^x[\sum_{i=0}^{\sigma -1}(\bar\psi-\psi_0^i)]}
+A\left(\tfrac{R_0}{R}\right)^{2+\alpha}\\
&\lesssim
AR^{-2}R_0^{2-\delta}+A\left(\tfrac{R_0}{R}\right)^{2+\alpha}
\lesssim
AR^{-2-\alpha\delta\gamma}R_0^{2}.
\end{align*}	
Let $\tau_R=\min\{n\ge 0:X_n\notin B_R\}$ and set
\[
w(x)=v(x)+C_1AR^{-2-\alpha\delta\gamma}E_\omega^x[\tau_R].
\]
Then, for any good point $x\in B_{R-R_0}$, by choosing $C_1$ big enough,
\[
E_\omega^x[w(X_\sigma)-w(x)]=E_\omega^x[v(X_\sigma)-v(x)]-C_1AR^{-2-\alpha\delta\gamma}E_\omega^x[\sigma]<0,
\]
where we used the fact that $E_\omega^x[\sigma]\ge R_0^2$. This implies
\begin{equation}\label{eq:subdiff-empty}
\partial w(x;B_R)=\emptyset \quad\text{ for any good point }x\in B_{R-R_0}
\end{equation}
where $\partial w(x;B_R)$ denote the sub-differential set of $w$ at $x$ with respect to $B_R$. 
For the definition of the sub-differential set and the ABP inequality, we refer to \cite[Definition 2.2, Lemmas 2.3 and 2.4]{GPT-19}. Next, we will apply the ABP inequality to bound $|v|$ from the above.

By \cite[Lemma 2.4]{GPT-19}, since
\[
L_\omega w=L_\omega v-C_1AR^{-2-\alpha\delta\gamma}
\lesssim AR^{-2},
\] 
we know that $|\partial w(x;B_R)|\lesssim A^dR^{-2d}$ for $x\in B_R$. Let 
\[
\ms B_R=\ms B_R(\omega,\gamma):=\#\text{bad points in }B_{R-R_0},
\]
where $\#S$ denotes the cardinality of a set $S$.
Display \eqref{eq:subdiff-empty} then yields
\[
|\partial w(B_R)|\lesssim 
[\ms B_R+\#(B_R\setminus B_{R-R_0})]A^dR^{-2d}
\lesssim
(\ms B_R+R^{d-1+\gamma})A^dR^{-2d}.
\]
Hence, by the ABP inequality \cite[Lemma 2.4]{GPT-19},
\[
\min_{B_R}w\ge -CR|\partial w(B_R)|^{1/d}
\gtrsim
-A(R^{-1}\ms B_R^{1/d}+R^{-(1-\gamma)/d}).
\]
Therefore, noting that $\max_{x\in B_R}E_\omega^x[\tau_R]\le (R+1)^2$ and choosing $\delta<1/d$,
\[
\min_{B_R} v\ge \min_{B_R}w-CAR^{-\alpha\delta\gamma}
\gtrsim
-A(R^{-1}\ms B_R^{1/d}+R^{-\alpha\gamma\delta}).
\]
Similar bound for $\min_{B_R}(-v)$ can be obtained by substituting $f,g$ by $-f,-g$ in the problem. Therefore
\[
\max_{B_R}|v|\lesssim A(R^{-1}\ms B_R^{1/d}+R^{-\alpha\gamma\delta}).
\]
\item Combining results in Steps~\ref{item:step1} and \ref{item:step2}, we get
\[
\max_{B_R}|u-\bar u_R|\lesssim A(R^{-1}\ms B_R^{1/d}+R^{-\alpha\gamma\delta}).
\]
It is shown in Step 6 of \cite[Proof of Theorem~1.5]{GPT-19} that, with
\begin{equation}\label{eq:def-x}
\ms X=\ms X(\omega):=\max_{R\ge 1}R^{-\gamma}\ms B_R^{1/d},
\end{equation}
we have $\mb E[\exp(c\ms X^d)]<\infty$. Therefore, recalling the values of $\gamma, A$ in \eqref{eq:def-gamma}, \eqref{eq:regularity}, we conclude that
\begin{equation}
\label{eq:quan-hom-good-reg}
\max_{B_R}|u-\bar u_R|
\lesssim
R^{-\alpha\gamma\delta}(1+R^{-q/d}\ms X)(|f|_{0,\alpha;\B_1}\norm{\psi_0}_\infty+|g|_{2,\alpha;\partial\B_1}).
\end{equation}
\end{enumerate}
\end{proof}

In what follows, we will relax the regularity of $g$ to be $C^{0,\alpha}(\partial\B_1)$.

\begin{proof}
[Proof of Proposition~\ref{thm:quant-homo}:]
First, we consider the case $y=0$. The function $g\in C^{0,\alpha}(\partial\B_1)$ can be extended into $\R^d$ so that $g\in C^{0,\alpha}(\R^d)$ and 
\[
|g|_{0,\alpha;\B_2}\le C|g|_{0,\alpha;\partial\B_1}.
\] 
We can further obtain a smooth perturbation of it. To this end, let $\rho\in C^\infty(\R^d)$ be a {\it mollifier} supported on $\B_1$ with $\int_{\B_1}\rho\dd x=1$. For $h\in(0,1)$, set $\rho_h(x):=h^{-d}\rho(\tfrac xh)$, and let $g_h=\rho_h*g$. That is, 
$g_h(x)=\int_{\R^d}\rho_h(x-z)g(z)\dd z$. Then $g_h$ satisfies
\[
\left\{
\begin{array}{lr}
&\abs{g-g_h}_{0;\partial\B_1}\le Ch^\alpha|g|_{0,\alpha;\partial\B_1}\\
&|g_h|_{2,\alpha;\partial\B_1}\le Ch^{-2}|g|_{0,\alpha;\partial\B_1}.
\end{array}
\right.
\]

Next, for $h\in(0,1)$, let $v:\bar B_R\to\R$ and $\bar v:\bar\B_1\to\R$ be solutions of 
\[
\left\{
\begin{array}{lr}
\tfrac{1}{2}\tr(\omega\nabla^2 v)=\tfrac{1}{R^2}f(\tfrac{x}{R})\psi(\theta_x\omega) &\text{ in }B_R\\
v(x)=g_h(\tfrac{x}{R}) &\text{ for }x\in\partial B_R
\end{array}
\right.
\]
and
\[
\left\{
\begin{array}{lr}
\tr(\bar aD^2\bar v)=f\bar\psi &\text{ in }\B_1\\
\bar v=g_h &\text{ on }\partial\B_1.
\end{array}
\right.
\]
Then, $\max_{\B_1}|\bar u-\bar v|\le \max_{\partial\B_1}|g-g_h|\le h^\alpha|g|_{0,\alpha;\partial\B_1}$, and
\[
\max_{B_R}|u-v|\le\max_{\partial B_R}|g(\tfrac{x}{R})-g_h(\tfrac{x}{R})|\le Ch^\alpha|g|_{0,\alpha;\partial\B_1}.
\]
 Moreover, by \eqref{eq:quan-hom-good-reg}, with $A_1=\norm{f}_{C^{0,\alpha}(\B_1)}\norm{\tfrac{\psi}{\tr(\omega)}}_\infty+[g]_{C^{0,\alpha}(\partial\B_1)}$ as in Proposition~\ref{thm:quant-homo},
\begin{align*}
\max_{x\in B_R}|v(x)-\bar v(\frac{x}{R})|
&\lesssim
R^{-\alpha\gamma\delta}(1+R^{-q/d}\ms X)(|f|_{0,\alpha;\B_1}\norm{\tfrac{\psi}{\tr(\omega)}}_\infty+|g_h|_{2,\alpha;\partial\B_1})\\
&\lesssim
A_1h^{-2}
R^{-\alpha\gamma\delta}(1+R^{-q/d}\ms X).
\end{align*}	
Notice that up to an additive constant, we may assume that $\inf_{\partial B_1}g=0$, so that $|g|_{0,\alpha;\partial\B_1}\le C[g]_{0,\alpha;\partial\B_1}$.
Therefore, putting $h=R^{-\alpha\gamma\delta/3}$, by the triangle inequality, 
\begin{equation}
\label{eq:quant-hom-less-reg}
\max_{x\in B_R}|u-\bar u(\tfrac{x}{R})|\lesssim 
A_1R^{-\alpha\gamma\delta/3}(1+R^{-q/d}\ms X).
\end{equation}
We proved Proposition~\ref{thm:quant-homo} for the case $y=0$ with  $\beta=\gamma\delta/3$.

Finally, for any $y\in B_{3R}$, it follows from \eqref{eq:quant-hom-less-reg} that
\[
\max_{x\in B_R(y)}\Abs{\bar u(\tfrac{x-y}{R})-u(x)}\lesssim
A_1R^{-\alpha\beta}(1+\ms X(\theta_y\omega) R^{-q/d}).
\]
Let
\[
\ms B_R(y)=\ms B_R(\theta_y\omega,\gamma):=\#\text{bad points in }B_{R-R_0}(y).
\]
Observe that $\ms B_R(y)\le \ms B_{4R}(0)$. Thus, recalling the definition of $\ms X$ in \eqref{eq:def-x},
\[
\ms X(\theta_y\omega)\le \max_{R\ge 1}R^{-\gamma}\ms B_{4R}^{1/d}\le 4^\gamma\ms X.
\]
Our proof of Proposition~\ref{thm:quant-homo} is complete.
\end{proof}

\subsection{Proofs of Lemmas \ref{lem:eta} and \ref{lem:test-exponential}}\label{asection:test-fun}

\begin{proof}
[Proof of Lemma~\ref{lem:eta}:]
By direct computation, for any $i=1,\ldots,d$, $x\in\Z^d$,
\begin{equation}
\label{eq:2nd-der-eta}
\partial_i^2\eta(x)=\theta(1+|x|^2)^{-\theta-2}[4(\theta+1)x_i^2-2|x|^2-2],
\end{equation}
\[
|\partial_i^3\eta(x)|
\le C\theta^3|x|(1+|x|^2)^{-\theta-2}.
\]
Moreover, noting that for any $y\in\R^d$ and $x\notin B_{\theta}$ with $|y-x|\le 1$, 
\[
|\partial_i^3\eta(y)|\le C\theta^3(\tfrac{1+\theta}{\theta})^{\theta+2}|x|(1+|x|^2)^{-\theta-2}\le C\theta^3|x|(1+|x|^2)^{-\theta-2}
\]
and $|\nabla_i^2\eta(x)-\partial_i^2\eta(x)|\le C\sup_{y:|y-x|\le 1}|\partial_i^3\eta(y)|$, 
we have, for a sufficiently large constant $C_0(\kappa)>1$  and $x\notin B_{C_0\theta^2}$, 
\begin{align*}
L_\omega\eta(x)
&\ge 
\sum_{i=1}^d\frac{\omega_i(x)}{2\tr\omega(x)}\left[\partial_i^2\eta(x)-C\theta^3|x|(1+|x|^2)^{-\theta-2}\right]\\
&\stackrel{\eqref{eq:2nd-der-eta}}{\ge}
\theta(1+|x|^2)^{-\theta-2}[4\kappa(\theta+1)|x|^2-|x|^2-1-C\theta^2|x|]>0.
\end{align*}	
On the other hand, for $x\in B_{C_0\theta^2}$, clearly $L_\omega\eta(x)\ge -\eta(x)\ge -1$. 

The lemma is proved.
\end{proof}

\begin{proof}[Proof of Lemma~\ref{lem:test-exponential}]
Computations show that, for $x\neq 0$, $i=1,\ldots, d$,
\[
\partial_i^2 e^{-2\alpha|x|/R}=\left(-\tfrac {2\alpha}{|x|}+\tfrac{2\alpha x_i^2}{|x|^3}+\tfrac{4\alpha^2x_i^2}{R|x|^2}\right)R^{-1}e^{-2\alpha|x|/R},
\]
\[
\partial_i^3 e^{-2\alpha|x|/R}
=
\left(\frac{6\alpha}{R}+\frac{3}{|x|}-\frac{4\alpha x_i^2}{R|x|^2}-\frac{4\alpha^2x_i^2}{R^2|x|}-	\frac{3x_i^2}{|x|^3}\right)
\frac{2\alpha x_i}{R|x|^2}
e^{-2\alpha|x|/R}.
\]
Note that, for $i=1,\ldots, d$, $x\in B_R\setminus B_{R/2}$,
\[
\Abs{(\partial_i^2-\tfrac12\nabla_i^2)e^{-2\alpha|x|/R}}
\le 
C\sup_{y:|y-x|\le 1}\Abs{\partial_i^3 e^{-2\alpha|y|/R}}
\le \frac{C\alpha}{R^3}e^{-2\alpha|x|/R},
\]
and so 
\begin{equation}
\label{eq:disc-L-error}
\Abs{\sum_{i=1}^d\omega(x,x+e_i)(\partial_i^2-\tfrac{1}{2}\nabla_i^2)e^{-2\alpha|x|/R}}\le \frac{C\alpha}{R^3}e^{-2\alpha|x|/R}.
\end{equation}
Further, by taking $K>0$ sufficiently large (note $R\ge K$), and choosing $\alpha>0$ to be sufficiently small, we have, for  $x\in B_R\setminus B_{R/2}$,
\begin{align*}
&\sum_{i=1}^d\omega(x,x+e_i)\partial_i^2e^{-2\alpha|x|/R}\\
&=\sum_{i=1}^d\omega(x,x+e_i)\left(-\tfrac {2\alpha}{|x|}+\tfrac{2\alpha x_i^2}{|x|^3}+\tfrac{4\alpha^2x_i^2}{R|x|^2}\right)R^{-1}e^{-2\alpha|x|/R}\\
&\le 
\left(
-\tfrac{\alpha}{|x|}+\tfrac{(1-2\kappa)\alpha|x|^2}{|x|^3}+\tfrac{2(1-2\kappa)\alpha^2|x|^2}{R|x|^2}
\right)R^{-1}e^{-2\alpha|x|/R}\\
&\le C(-1+C\alpha)\tfrac{\alpha}{R^2}e^{-2\alpha|x|/R}
\le -\tfrac{C\alpha}{R^2}e^{-2\alpha|x|/R}.
\end{align*}
This, together with \eqref{eq:disc-L-error}, implies, 
\[
L_\omega e^{-2\alpha|x|/R}\le -\tfrac{C\alpha}{R}e^{-2\alpha|x|/R}, \quad \text{for }x\in B_R\setminus B_{R/2}.
\]
Display \eqref{eq:20200203-1} is proved.

To prove \eqref{eq:L-e^x2}, note that when $x=0$, $L_\omega(e^{-A|x|^2})=e^{-A}-1>-1$. When $x\in\Z^2\setminus\{0\}$, choosing $A>0$ sufficiently large, 
\begin{align*}
L_\omega(e^{-A|x|^2})
&=
e^{-A|x|^2}\left[\sum_{i=1}^d\omega(x,x+e_i)(e^{2Ax_i-A}+e^{-2Ax_i-A})-1\right]\\
&\ge 
e^{-A|x|^2}[\kappa e^{2A-A}-1]>0.
\end{align*}	
Display \eqref{eq:L-e^x2} is proved. 

It remains to prove \eqref{eq:L-e^x2/R2}. Using the inequalities $e^a+e^{-a}\ge 2+a^2, e^a\ge 1+a$, we get, by taking $A$ sufficiently large, for $x\in B_R\setminus B_{R/2}$, 
\begin{align*}
L_\omega(e^{-A|x|^2/R^2})
&=e^{-A|x|^2/R^2}\sum_{i=1}^d\omega(x,x+e_i)\left[e^{-A(1+2x_i)/R^2}+e^{-A(1-2x_i)/R^2}-2\right]\\
&\ge 
e^{-A|x|^2/R^2}\sum_{i=1}^d\omega(x,x+e_i)\left[
e^{-A/R^2}(2+4A^2x_i^2/R^4)-2
\right]\\
&\ge 
\tfrac{A}{R^2}e^{-A|x|^2/R^2}\sum_{i=1}^d\omega(x,x+e_i)
\left[\tfrac{4Ax_i^2}{R^2}(1-\tfrac{A}{R^2})-1\right]\\
&\ge 
\tfrac{A}{R^2}e^{-A|x|^2/R^2}\left[\frac{2\kappa A|x|^2}{R^2}-1\right]>0.
\end{align*}	
Our proof is complete.
\end{proof}

\end{document}